\documentclass[twocolumn]{autart}    % Enable this line and disable the 
\usepackage{amsmath,amssymb,amsfonts}
\usepackage{algorithm}
\usepackage{algpseudocode}
\usepackage{graphicx}
\usepackage{algorithm}
\usepackage{hyperref}
\usepackage{textcomp}
\usepackage{amsmath}
\usepackage[normalem]{ulem}
\usepackage{bm}
\usepackage{mathrsfs}
\usepackage[utf8]{inputenc}
\usepackage[titletoc]{appendix}
\usepackage[authoryear]{natbib}
\usepackage{color}
\usepackage{hyperref}



\newtheorem{lemmax}{Lemma}[section]
\begin{document}

\begin{frontmatter}
                                              % is not shown in the output.

\title{Asymptotically Efficient Recursive Identification Under One-Bit Communications Achieving Original CRLB \thanksref{footnoteinfo}} % Title, preferably not more 

\thanks[footnoteinfo]{This paper was not presented at any IFAC 
meeting. Corresponding author: Yanlong Zhao.
This work was supported by the National Natural Science Foundation of China under Grants 62588101, 62401340, and T2293773, and the Natural Science Foundation of Shandong Province under Grant ZR2023QF103. }

\author[Paestum,Rome]{Xingrui Liu}\ead{liuxingrui@amss.ac.cn},    % Add the 
\author[Baiae]{Jieming Ke}\ead{kejieming@amss.ac.cn},               % e-mail address 
\author[Paestum,Rome]{Mingjie Shao}\ead{mingjieshao@amss.ac.cn},   
\author[Paestum,Rome]{Yanlong Zhao}\ead{ylzhao@amss.ac.cn}  % (ead) as shown

\address[Paestum]{State Key Laboratory of Mathematical Sciences, Academy of Mathematics and Systems Science, Chinese Academy of Sciences, Beijing 100190, China}  % Please supply                                              
\address[Rome]{School of Mathematical Sciences, University of Chinese Academy of Sciences, Beijing 100049, China}             % full addresses
\address[Baiae]{Department of Information Engineering, University of Padova, Padova 35131, Italy}        % here.

                             % keyword list or with the 
                                          % help of the Automatica 
                                          % keyword wizard

\begin{abstract}
This paper develops an asymptotically efficient recursive identification algorithm for autoregressive systems with exogenous inputs under one-bit communications.
In particular, the proposed method asymptotically achieves the Cramér--Rao lower bound (CRLB) based on the original data before quantization (original CRLB), whereas existing approaches typically attain only the CRLB corresponding to the quantized observations.
The primary reason is that the existing methods quantize only the current system output, resulting in non-negligible information loss under one-bit quantization.
To overcome this challenge, we present a novel quantization method that integrates both current and historical system outputs and inputs to provide richer parameter information in one-bit data, allowing the information loss caused by quantization to become a minor term relative to the original CRLB.
Based on this technique, a corresponding remote estimation algorithm is further proposed.
To address the convergence analysis challenge posed by the non-independence of the one-bit data, we establish a new framework that analyzes the tail probability of integrated data formed by combining current and historical system outputs and inputs before quantization, thereby eliminating the need for the traditional independence assumption on the quantized data.
It is proven that the remote estimate achieves asymptotic normality, and the error covariance matrix converges to the original CRLB, confirming its asymptotic efficiency.
Compared to existing identification algorithms under one-bit data, this method reduces the asymptotic mean squared error by at least $1 - 2/\pi \approx 36\%$.
Several numerical examples are simulated to show the effectiveness of the proposed algorithm.
\end{abstract}

\end{frontmatter}

\section{Introduction}

\subsection{Background and motivation}

Over the past two decades, networked control systems have been widely adopted across various domains, including industrial automation \citep{Akhtar2023Developing}, wireless sensor networks \citep{jiang2024linear}, telecommand applications \citep{Mashiko2025}, and vehicle platooning \citep{zhang2025decentralized}. 
In such systems, tasks such as data acquisition, processing, and decision-making could be distributed across multiple nodes, significantly enhancing flexibility and scalability \citep{Ribeiro2006}.
However, a critical limitation lies in the constrained communication bandwidth between nodes, which often prevents remote estimators from directly accessing system inputs and outputs \citep{You2015}. 
Instead, these estimators must rely on low-bit data transmission from other nodes for system identification \citep{Wang2026}, a common scenario in practical applications:

\romannumeral1) Industrial automation \citep{Vergallo2023}: 
Internet-of-things and cloud computing technologies enable remote monitoring and diagnostics of complex industrial systems, allowing for timely detection and resolution of issues. 
However, managing the massive volume of data generated by devices within the constraints of limited bandwidth presents a significant challenge.

\romannumeral2) Wireless sensor networks \citep{Li2007}: 
In sensor networks, geographically distributed sensors collaborate to achieve system identification.  
Constraints such as high sensor costs, limited power supplies, and the significant energy consumption of communication severely limit the available communication capacity.

\romannumeral3) Telecommand applications \citep{Almohamad2025}: 
Accurate estimation of satellite attitude is essential for ground control operations. 
Nevertheless, such estimations need to be performed under stringent communication bandwidth constraints, necessitating efficient data transmission strategies.

These scenarios highlight a fundamental challenge: how to optimize estimation accuracy under strict communication resource constraints. 
In this context, the error covariance matrix serves as a crucial metric for quantifying estimation accuracy \citep{Guo2026}. 
To establish theoretical performance limits, the Cramér-Rao lower bound (CRLB) provides a fundamental benchmark \citep{Friedlander1984}, representing the minimum achievable error covariance matrix for any unbiased estimator \citep{Gustafsson2009}. 
Building upon this theoretical framework, this study aims to investigate the optimal identification problem under one-bit communications, and develop a recursive identification algorithm whose error covariance matrix converges to the CRLB, namely an asymptotically efficient recursive identification algorithm.

\subsection{Related literature}

Numerous studies have explored methods to optimize parameter estimation accuracy under one-bit data.
Within a transmission framework that utilizes a quantizer with a fixed threshold, various techniques have been proposed, including the expectation maximization method \citep{godoy2011identification, Risuleo2020, Dario2022}, the empirical measure method \citep{yin2007asymptotically}, and the quasi-Newton method \citep{wang2024asymptotically}. 
For any given threshold, the error covariance matrix of these methods converges to the CRLB for that fixed threshold. 
However, selecting the optimal threshold that minimizes the CRLB in advance remains challenging due to its dependence on the unknown system parameter \citep{wang2024threshold}, which hinders efforts to minimize the error covariance matrix within this transmission framework.

To further enhance estimation accuracy, significant attention has been given to transmission frameworks that employ adaptive quantizers due to their flexibility.
In particular, \citet{you2015recursive} introduced a co-design approach for an adaptive quantizer and a remote estimator, where the quantizer updates its threshold based on parameter estimates from the previous time.
This adaptive adjustment enables the threshold in \citet{you2015recursive} to approximate the optimal adaptive threshold \citep{wang2024threshold}.
Consequently, the error covariance matrix of the algorithm proposed by \citet{you2015recursive}  converges to the CRLB for the optimal adaptive threshold, outperforming methods that rely on fixed quantizers.

\citet{you2015recursive} and \citet{wang2024threshold} have achieved optimal estimation accuracy when quantizing only the current system output.
Nevertheless, the information loss caused by quantization in such transmission frameworks leads to asymptotic covariance matrices that are at least \(\pi/2 \approx 1.56\) times greater than the CRLB based on the original data before quantization (original CRLB) \citep{you2015recursive, wang2024threshold}.

\begin{figure*}[h]
\centering
\noindent\includegraphics[width=0.85\textwidth]{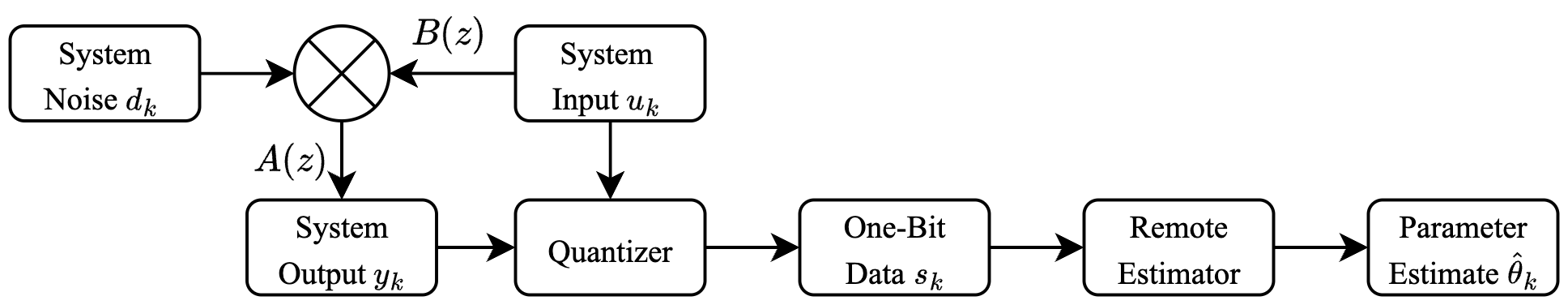}
\caption{Block diagram of the identification problem for ARX systems under one-bit communications}
\label{fig}
\end{figure*}

Moreover, in the pursuit of optimal estimation accuracy, existing recursive approaches \citep{yin2007asymptotically, wang2024asymptotically, you2015recursive}  impose stringent requirements on both the system model and the input type. 

For the system model, these algorithms are primarily designed for finite impulse response (FIR) systems and are unsuitable for optimal identification in more complex dynamic systems, such as autoregressive systems with exogenous inputs (ARX systems).
The extension from FIR systems to ARX systems is difficult because the estimator they consider receives only one-bit data quantized from the current system output, thereby failing to capture the complete system regression vector.
While some studies have explored recursive identification for ARX systems in such transmission frameworks \citep{Casini2011,csaji2012recursive}, these methods face challenges in achieving asymptotic efficiency, leaving the development of asymptotically efficient recursive identification algorithms for ARX systems under one-bit communications as an open problem.

For the input type, these analyses primarily develop along two independent directions.
The methods based on fixed quantizers are confined to deterministic input scenarios \citep{yin2007asymptotically, wang2024asymptotically}, while adaptive quantizer mainly focuses on stochastic conditions \citep{you2015recursive}.
In practice, many real-world applications, such as modal analysis of mechanical or civil engineering structures, involve mixed input signals that combine deterministic control commands and stochastic noise \citep{Reynders2008}. 
This highlights the need for a unified framework capable of handling both deterministic and stochastic inputs.

Motivated by these challenges, this paper seeks to develop a novel quantization method and a corresponding remote estimator to achieve the original CRLB, with applicability to ARX systems under both deterministic and stochastic inputs.

\subsection{Main contributions}

This paper investigates the optimal recursive identification problem for ARX systems under one-bit communication constraints.
To overcome the fundamental information loss induced by one-bit quantization, a novel joint design of a quantizer and a remote estimator is proposed.
At the quantizer side, parametric information is first extracted from the original measurements by the asymptotically efficient recursive least-squares (RLS) algorithm \citep{Liu2025}.
The quantizer then transmits one-bit information formed as the sign of the difference between the current remote estimate and the local RLS estimate.
The remote estimator completes the online estimation process based on the received one-bit data using a stochastic approximation (SA) scheme, which drives the remote estimate to track the local RLS estimate.
The main contributions can be summarized as follows:

\romannumeral1) This work develops an asymptotically efficient recursive identification algorithm under one-bit communications, whose error covariance matrix converges to the original CRLB before quantization.
By extracting parametric knowledge from both current and historical system outputs and inputs into one-bit data, the proposed approach reduces the asymptotic mean square error (MSE) by at least \(1 - 2/\pi \approx 36\%\) compared with existing asymptotically efficient algorithms that quantize only the current system output \citep{yin2007asymptotically, wang2024asymptotically, you2015recursive}.
Besides, the remote estimate is proved to converge to the true parameter in both almost sure and \(L^{p}\) sense for any positive integer \(p\).

\romannumeral2) This work extends the system model to ARX systems, in contrast to existing asymptotically efficient recursive identification algorithms that solely focus on FIR systems \citep{yin2007asymptotically, wang2024asymptotically, you2015recursive}.
The primary difficulty arises because the ARX system regressors involve the system outputs that are unavailable in the remote estimator. 
The proposed algorithm overcomes this challenge by pre-processing the data using the regressors before quantization, thereby embedding the necessary regressor information in the one-bit data and circumventing the need to rely on unavailable regressors at the remote estimator.

\romannumeral3) This work establishes a novel convergence analysis framework tailored to the non-independent one-bit data.
This framework analyzes the tail probability of integrated data formed by combining current and historical system outputs and inputs before quantization, thereby eliminating the need for the traditional independence assumption on the quantized data \citep{zhang2019asymptotically, wang2023identification, wang2024asymptotically}.
In addition to addressing the non-independent challenge, this framework accommodates both deterministic and stochastic inputs and does not rely on the common assumption in quantized recursive identification that the parameter resides within a known compact set.

\subsection{Overview and notations}

The remainder of this paper is organized as follows. 
Section \ref{sec_a} formulates the problem with ARX system structures and introduces the fundamental assumptions.
Section \ref{sec_b} focuses on the algorithm construction.
Section \ref{sec_c}  establishes the main results, where
Section \ref{sec_c1} analyzes the convergence rate for the tail probability of the difference between the remote estimate and the parametric knowledge,
Section \ref{sec_c2} demonstrates both almost sure and $L^{p}$ convergence, and
Section \ref{sec_c3} confirms the asymptotic efficiency of the remote estimate.
Section \ref{sec_e} gives simulation examples to verify the conclusion.
Section \ref{sec_f} is the summary and prospect of this paper.

In this paper, $\mathbb{R}$, $\mathbb{R}^{n}$, $\mathbb{R}^{n \times n}$, and $\mathbb{Z}$  are the sets of real numbers, $n$-dimensional real vectors and $n$-order matrices, and integers, respectively. 
For a constant $x$, $\vert x \vert$ denotes its absolute value.
For positive integers $m$ and $n$, $\mathrm{mod}(m,n)$ denotes the remainder of $m$ divided by $n$.
For a vector $a=[a_{1},a_{2},\ldots ,a_{n}] \in \mathbb{R}^{n}$, $a_{i}$ denotes its $i$-th element, $a^{T}$ denotes its transpose; $\Vert a \Vert$ denotes its Euclidean norm, i.e., $\Vert a \Vert = (\sum_{i=1}^{n} a_{i}^{2})^{1/2}$.
For a matrix $A = [a_{i,j}]$, $i = 1,2,\ldots,n$, $j=1,2,\ldots,n$, $A^{T}$ denotes its transpose; $\text{tr}(A)$ denotes its trace; $\lambda_{\max}(A)$ and $\lambda_{\min}(A)$ denote its maximum eigenvalue and minimum eigenvalue respectively; $\Vert A \Vert$ denotes its $2$-norm, i.e., $\Vert A \Vert = \sqrt{ \lambda_{\max}(A^{T}A)}$. 
$I_{n}$ is an $n$-dimensional identity matrix.
For a set $\Omega$, $\Omega^{\mathrm{c}}$ denotes its complement.
$\mathbb{P}$ denotes the probability operator.
$\mathbb{E}$ denotes the expectation operator.
The function $I_{\{\cdot\}}$ denotes the indicator function, whose value is $1$ if its argument (a formula) is true and $0$ otherwise.
$\mathcal{N}(\mu,\sigma^2)$ denotes the Gaussian distribution with mean $\mu$ and variance $\sigma^2$.

\section{Preliminaries} \label{sec_a}

\subsection{Problem formulation}

Consider an ARX system described by
\begin{align}\label{model_a}
A(q)y_{k} =  B(q)u_{k} + d_{k},\quad k \geq 1,
\end{align}
where $k$ is the time index; the system input $u_{k} \in \mathbb{R}$ and output $y_{k} \in \mathbb{R}$ are both known to the quantizer; $d_{k} \in \mathbb{R}$ is the observation noise; $A(q)$ is the $m$-th order polynomial and $B(q)$ is the $n$-th order polynomial, both expressed in terms of unit backward shift operator: $q^{-1}y_{k} = y_{k-1}$ as $A(q) \triangleq 1+ a_{1}q^{-1}+ \ldots + a_{m}q^{-m}$ and $B(q) \triangleq b_{1}q^{-1} + \ldots + b_{n}q^{-n}$; $a_{1}, \ldots, a_{m} \in \mathbb{R}$ and $b_{1}, \ldots, b_{n} \in \mathbb{R}$ are the unknown parameters; the remote estimator relies on one-bit $s_{k}$ from a quantizer to identify the system, as depicted in Fig.\ref{fig}.
Besides, we stipulated such that $y_k = 0$ for $k \leq 0$ and $u_k = 0$ for $k < 0$.

The objective of this paper is to co-design a quantizer and a remote estimator for the recursive identification of the unknown parameter \( \theta \triangleq [a_{1}, \ldots, a_{m}, b_{1}, \ldots, b_{n}]^{T} \) with the goal of optimizing estimation accuracy.

To formalize the fundamental performance limit, we first introduce the original CRLB, which characterizes the minimum achievable error covariance matrix for any unbiased estimator with the original measurements $\{y_{k}\}_{k=1}^{\infty}$.
Specifically, for any unbiased estimator $\hat{\theta}_{k}$ of $\theta$, there exists a lower bound $\Sigma_{\mathrm{CR}}(k)$ satisfying
\[
\mathbb{E}\left[\tilde{\theta}_{k}\tilde{\theta}_{k}^{T}\right] \geq \Sigma_{\mathrm{CR}}(k),
\]
where $\tilde{\theta}_{k} \triangleq \hat{\theta}_{k}-\theta$ denotes the estimation error and $\Sigma_{\mathrm{CR}}(k)$ is called the original CRLB at time $k$.

\begin{rem}
The CRLB is also adopted as a fundamental benchmark for evaluating the estimation accuracy of asymptotically unbiased estimators.
Although the finite-sample error covariance may fall below the CRLB, the asymptotic error covariance cannot be smaller than the limit of the CRLB except on a Lebesgue measure zero subset of the parameter space \citep{Van2000}.
That is, for any asymptotically unbiased estimator $\hat{\theta}_{k}$ of $\theta$,
\[
\liminf_{k \to \infty}k\mathbb{E}\left[\tilde{\theta}_{k}\tilde{\theta}_{k}^{T}\right] \geq \bar{\Sigma}_{\mathrm{CR}},  \quad \text{for almost every} \,\, \theta,
\]
where $\bar{\Sigma}_{\mathrm{CR}} = \lim_{k \to \infty}k\Sigma_{\mathrm{CR}}(k)$.
\end{rem}

Hence, our goal is to ensure that the covariance matrix of the estimation error attains the original CRLB, achieving asymptotic efficiency, which is commonly defined in two distinct ways:
\begin{defn}\label{def_asy1} Asymptotic efficiency in the distribution sense \citep{Fabian1978,you2015recursive}:
An estimate $\hat{\theta}_{k}$ of $\theta$ at time $k$ is asymptotically efficient if
\begin{align}\label{def_112}
\sqrt{k}\tilde{\theta}_{k} \xrightarrow{d}  \mathcal{N}\left(0,  \bar{\Sigma}_{\mathrm{CR}} \right).
\end{align}
\end{defn}

\begin{defn}\label{def_asy2} Asymptotic efficiency in the covariance sense \citep{Rao1961,wang2024asymptotically}:
An estimate $\hat{\theta}_{k}$ of $\theta$ at time $k$ is asymptotically efficient if
\begin{align}\label{def_1122}
\lim\limits_{k \rightarrow \infty} k\mathbb{E}\left[\tilde{\theta}_{k}\tilde{\theta}_{k}^{T} \right]=   \bar{\Sigma}_{\mathrm{CR}}.
\end{align}
\end{defn}

\begin{rem}
The concepts of asymptotic efficiency based on Definitions \ref{def_asy1} and \ref{def_asy2} are not mutually inclusive. On one hand, as stated in Appendix 9B of \citet{ljung1987theory}, an estimate that satisfies (\ref{def_112}) does not necessarily satisfy (\ref{def_1122}). On the other hand, an estimate satisfying (\ref{def_1122}) is not necessarily normally distributed, meaning it does not necessarily fulfill (\ref{def_112}).
\end{rem}

We aim to develop a recursive identification algorithm that is asymptotically efficient in both the distribution and covariance senses, enabling the covariance matrix of the estimation error to attain the original CRLB asymptotically.

\subsection{Assumptions}
To proceed with our analysis, we introduce some fundamental assumptions concerning the system model, noises, inputs, and outputs, which are primarily consistent with those in \citep{ljung1987theory}.
First, we provide the following two definitions for assumptions concerning the system outputs and inputs.

\begin{defn}\label{def_a}
For a process $\{ x_{k} \}_{k=1}^{\infty}$, the operator $\bar{\mathbb{E}}$ denotes
\[
\bar{\mathbb{E}}\left[\{  x_{k} \}_{k=1}^{\infty}\right] \triangleq \lim\limits_{k \rightarrow \infty} \frac{1}{k}\sum_{l=1}^{k}\mathbb{E}\left[x_{l}\right].
\]
\end{defn}

\begin{defn}\label{def_b}
A process $\{ x_{k} \}_{k=1}^{\infty}$ is said to be quasi-stationary if it is subject to 
\begin{align*}
& \mathbb{E}\left[x_{k} \right] = m_{k}, \quad \mathop{\sup}\limits_{k} \left\vert m_{k} \right\vert < C, \quad k \geq 1, \\
& \mathbb{E}\left[x_{i}x_{j}\right] = R_{i,j}^{x}, \quad \mathop{\sup}\limits_{i,j} \left\vert R_{i,j}^{x} \right\vert < C,  \quad i,j \geq 1, 
\end{align*}
and the covariance function
\[
R_{\tau}^{x} \triangleq \bar{\mathbb{E}}\left[\{  x_{k} x_{k-\tau}\}_{k=1}^{\infty}\right],   \quad \tau \in \mathbb{Z},
\]
exists, where $C$ is a positive constant.
\end{defn}

\begin{rem}
The definition of quasi-stationary processes, comprising both deterministic and stochastic components, is widely used in system identification \citep{ljung1987theory} and includes weakly stationary processes as a special case \citep{Park2018}.
On the one hand, quasi-stationary processes do not require the correlation function to be time-independent.
Instead, they only necessitate a convergence condition regarding the correlation. 
On the other hand, the mean of a quasi-stationary process is only required to satisfy a boundedness condition, rather than being strictly constant over time.
\end{rem}

\begin{assum}\label{ass_a}
The system input and output data set $Z^{\infty} \triangleq \{u_{0}, y_{1}, u_{1}, y_{2}, \ldots \}$ is generated by a bounded deterministic sequence $\{r_{k}\}_{k=0}^{\infty}$ and a stochastic sequence $\{e_{k}\}_{k=0}^{\infty}$ of independent random variables with zero mean values, variance $\delta_{e}^{2}$, and bounded moments of any positive integer order,
such that for some filters $\{ f_{i}^{(j)}(k) \}_{i=1}^{\infty}$, $j = 1,2,3,4$, $k \geq 1$,
\begin{align*}
& y_{k} = \sum_{i=1}^{\infty}  f_{i}^{(1)}(k)r_{k-i}  + \sum_{i=0}^{\infty} f_{i}^{(2)}(k)e_{k-i}, \\ 
& u_{k} = \sum_{i=0}^{\infty}  f_{i}^{(3)}(k)r_{k-i}  + \sum_{i=0}^{\infty} f_{i}^{(4)}(k)e_{k-i},
\end{align*}
where

\romannumeral1) Uniform stability: The family of filters $\{ f_{i}^{(j)}(k) \}_{i=1}^{\infty}$, $j = 1,2,3,4$, $k \geq 1$ is uniformly stable, i.e., there exists a filter $\{ f_{i} \}_{i=1}^{\infty}$ such that
$
\mathop{\sup}\limits_{k,j} \vert f_{i}^{(j)}(k) \vert \leq \vert f_{i} \vert
$
for all $i \geq 1$,
where $\sum_{i=1}^{\infty} \vert f_{i} \vert < \infty $;

\romannumeral2) Joint quasi-stationary:
The system output $\{y_{k}\}_{k=1}^{\infty}$ and the system input $\{u_{k}\}_{k=0}^{\infty}$ are both quasi-stationary, and, in addition, the cross-covariance function
\[
R_{\tau}^{yu} \triangleq \bar{\mathbb{E}}\left[ \{y_{k}u_{k-\tau}\}_{k=1}^{\infty}\right],  \quad \tau \in \mathbb{Z},
\]
exists;

\romannumeral3) Persistent excitation:
The system regression vector $\phi_{k} \triangleq [-y_{k-1}, \ldots, -y_{k-m}, u_{k-1}, \ldots, u_{k-n}]^{T}$ satisfies the condition that
\begin{align*}
& \bar{\mathbb{E}}\left[ \{ \phi_{k}\phi_{k}^{T}\}_{k=1}^{\infty}\right]   \\
= &
\begin{bmatrix}
R_{0}^{y} & R_{1}^{y} & \cdots  & R_{m-1}^{y} & -R_{0}^{uy} &\cdots  & -R_{n-1}^{yu}\\
R_{1}^{y} &R_{0}^{y} &\cdots  & R_{m-2}^{y} & -R_{1}^{uy} &\cdots  & -R_{n-2}^{yu}\\
\vdots & \vdots & \ddots &  \vdots &  \vdots & &  \vdots \\
R_{m-1}^{y} & R_{m-2}^{y} & \cdots &  R_{0}^{y} & -R_{m-1}^{uy} & \cdots &  -R_{0}^{uy} \\
-R_{0}^{uy} & -R_{1}^{uy} & \cdots &  -R_{m-1}^{uy} & R_{0}^{u} &\cdots  & R_{m-1}^{u} \\
\vdots & \vdots & &  \vdots &  \vdots &\ddots &  \vdots \\
-R_{n-1}^{yu} & -R_{n-2}^{yu} &\cdots  &-R_{0}^{uy} &R_{m-1}^{u} &\cdots & R_{0}^{u}
\end{bmatrix}
\end{align*}
 is a positive definite matrix.
\end{assum}

\begin{assum}
\label{ass_b}
The filter $A(q)$ has no poles on or outside the unit circle.
\end{assum}

\begin{assum}
\label{ass_c}
The observation noise $\{d_{k}\}_{k=1}^{\infty}$ is a sequence of independent Gaussian random variables with zero mean and variance $\delta_{d}^{2}$.
\end{assum}

\begin{rem}
Assumptions \ref{ass_a}-\ref{ass_c} are primarily for the conditions of Theorem 9.1 of \citet{ljung1987theory}, which represent common assumptions for the asymptotic expression of the error covariance matrix in classical identification algorithms.
The key difference in this work is the relaxation of the assumption that the system parameter resides within a known compact set.
This is achieved by extending \( \mathbb{E}[e_{k}^{8}] < \infty \) to \( \mathbb{E}[e_{k}^{p}] < \infty \) for any positive integer \( p \), which enables a more comprehensive analysis of the tail probability.
This extension has broad applicability, being suitable for many types of random variables, including but not limited to Gaussian, Laplace, Poisson, and bounded variables such as uniform variables. 
\end{rem}

\section{Algorithm design}\label{sec_b}

This section presents a recursive identification algorithm that achieves the original CRLB under one-bit communications. 
The key idea is to let the remote estimator track the RLS estimate, which is asymptotically efficient for an ARX system \citep{Liu2025}.
To implement this idea, the quantizer first locally computes the RLS estimate, which is given by
\begin{align}
& \hat{\theta}_{k}^{\text{RLS}} = \hat{\theta}_{k-1}^{\text{RLS}} + a_{k}P_{k-1}\phi_{k}\left(y_{k}-\phi_{k}^{T}\hat{\theta}_{k-1}^{\text{RLS}}\right), \label{RLS1} \\
& a_{k} = \left(1+\phi_{k}^{T}P_{k-1}\phi_{k}\right)^{-1}, \label{RLS2}  \\
& P_{k} = P_{k-1} - a_{k}P_{k-1}\phi_{k}\phi_{k}^{T}P_{k-1}, \label{RLS3} 
\end{align}
where the initial values $\hat{\theta}_0^{\text{RLS}} \in \mathbb{R}^{n+m}$, and $P_{0} \in \mathbb{R}^{(n+m)\times(n+m)}$ is a positive definite matrix.
Then, the quantizer generates the one-bit information:
\begin{align}
& s_{k} = I_{\{x_{k}^{T}\hat{\theta}_{k}^{\text{RLS}} > x_{k}^{T}\hat{\theta}_{k-1}\}} -  I_{\{x_{k}^{T}\hat{\theta}_{k}^{\text{RLS}} \leq x_{k}^{T}\hat{\theta}_{k-1}\}}, \label{Remote3} \\ 
& x_{k} = \epsilon_{\mathrm{mod}(k,m+n)+1}, \label{Remote2} 
\end{align}
where $s_{k}$ indicates the sign of the difference between the selected components of the local RLS estimate $\hat{\theta}_{k}^{\text{RLS}}$ and the remote estimate $\hat{\theta}_{k}$. Unit vector $\epsilon_{i}$ is used to select the \((\mathrm{mod}(k,m+n)+1)\)-th component at each time step \(k\).
It is defined such that its $i$-th element is $1$  and all remaining elements are $0$ for $i = 1,2,\ldots,n+m$.

Using this one-bit information, the remote estimator updates its estimate via a SA scheme
\begin{align}
 \hat{\theta}_{k} = \hat{\theta}_{k-1} + \beta_{k}x_{k}\frac{s_{k}}{k^{\alpha}}, \label{Remote1} 
\end{align}
where the initial value $\hat{\theta}_0 \in \mathbb{R}^{n+m}$, $1/2 < \alpha < 1$ and $0 < \uline{\beta} \leq \beta_{k} \leq \bar{\beta} < \infty$ are the step-size coefficients which can be designed.

Combining these considerations, we construct the proposed recursive RLS–SA estimator in Algorithm 1.

\begin{algorithm}
\caption{RLS–SA Algorithm}
\label{alg_b_pseudo}
\begin{algorithmic}[1]
\State \textbf{Require:} Initial estimates $\hat{\theta}_0\in\mathbb{R}^{n+m}$, $\hat{\theta}^{\mathrm{RLS}}_0\in\mathbb{R}^{n+m}$, and $P_0\in\mathbb{R}^{(n+m)\times(n+m)}$; step-size coefficients $0 < \uline{\beta} \leq \beta_{k} \leq \bar{\beta} < \infty$ and  $1/2<\alpha < 1$.
\State \textbf{Ensure:} Remote estimate $\hat{\theta}_k$.
\For{$k=1,2,\ldots$}
    \State \textbf{Quantizer side:}
    \State Compute RLS gain $a_k$ by (\ref{RLS2})
    \State Update RLS estimate $\hat{\theta}^{\mathrm{RLS}}_k$ by (\ref{RLS1})
    \State Update RLS covariance matrix $P_k$ by (\ref{RLS3})
    \State Generate $x_k$ by (\ref{Remote2})
    \State Generate one-bit message $s_k$ by (\ref{Remote3})
    \State Update remote estimate $\hat{\theta}_k$ by (\ref{Remote1})
    \State Transmit $s_k$ to remote estimator
    \State \textbf{Remote estimator side:}
    \State Receive one-bit message $s_k$ from quantizer
    \State Generate $x_k$ by (\ref{Remote2})
    \State Update remote estimate $\hat{\theta}_k$ by (\ref{Remote1})
\EndFor
\end{algorithmic}
\end{algorithm}

\begin{rem}
The coding problem for the time-varying RLS estimate considered in this work is fundamentally different from that for a time-invariant parameter \citep{Como2010Anytime, Shirazinia2011Anytime}.
Due to the presence of stochastic noise, the convergence rate of $\hat{\theta}_{k}^{\mathrm{RLS}} - \hat{\theta}_{k-1}^{\mathrm{RLS}}$ is only at a polynomial rate. 
Since the convergence rate of the tracking error $\hat{\theta}_{k}^{\mathrm{RLS}} - \hat{\theta}_{k}$ cannot be faster than that of $\hat{\theta}_{k}^{\mathrm{RLS}} - \hat{\theta}_{k-1}^{\mathrm{RLS}}$, it is therefore at most polynomial. 
Consequently, classical coding schemes designed for time-invariant parameters cannot achieve traditional exponential tracking performance in this setting, as they typically rely on exponential step-size decays. 
Moreover, the RLS estimation error sequence is non-independent, while existing theories typically model such errors as independent noise sequences \citep{yin2007asymptotically, wang2024asymptotically, you2015recursive}.
\end{rem}

\section{Main results}\label{sec_c}
This section analyzes the asymptotic efficiency of the RLS-SA algorithm and establishes its key convergence properties, including almost sure and \(L^{p}\) convergence.
To examine the difference between the remote estimate and the RLS estimate, we define the auxiliary variable $\omega_{k} \triangleq x_{k}^{T}\hat{\theta}_{k} - x_{k}^{T}\hat{\theta}^{\text{RLS}}_{k+m+n}$ and aim to conduct a detailed analysis of its tail probability.
This approach enables a precise characterization of the algorithm's performance, thereby avoiding the inaccuracies introduced by classical approximation methods such as modeling quantization noise as a Gaussian variable \citep{Dey2014}.
The analyzes proceed as follows:

\romannumeral1) The convergence rate of the tail probability of $\omega_{k}$ is first established.

\romannumeral2) Based on the tail probability analysis, we establish the almost sure convergence and \(L^{p}\) convergence for any positive integer $p$ of the remote estimate.  

\romannumeral3) Finally, the asymptotic efficiency of the RLS-SA algorithm is established.

\subsection{The difference between the remote estimate and the RLS estimate}\label{sec_c1}

\begin{thm}\label{theorem1}
Under Assumptions \ref{ass_a}-\ref{ass_c}, 
\begin{align}\label{main}
\mathbb{P}\left(\left\vert \omega_{k} \right\vert > \frac{\rho}{k^{\alpha}} \right) = O\left(\frac{1}{k^{p}}\right),
\end{align}
where  $p$ is an arbitrary positive integer and $\rho \triangleq \bar{\beta} + 2\uline{\beta}/3$.
\end{thm}

\begin{pf}
Note that, at the \(k\)-th step, the RLS-SA algorithm adjusts only the value of the \((\text{mod}(k, m+n) + 1)\)-th dimension.
Consequently, the values for $\hat{\theta}_{k-(m+n)}$, $\hat{\theta}_{k-(m+n)+1}$, $\hat{\theta}_{k-1}$ at the \((\text{mod}(k, m+n) + 1)\)-th dimension remain unchanged.
This leads to
\begin{align}\label{sss222333iii}
x_{k}^{T}\hat{\theta}_{k-1} = x_{k}^{T}\hat{\theta}_{k-(m+n)}.
\end{align}

Then, together with $\Vert x_{k}\Vert= 1$, $x_{k} = x_{k-(m+n)}$ and (\ref{RLS1}), it holds that
\begin{align*}
\omega_{k}
= & x_{k-(m+n)}^{T}\hat{\theta}_{k-(m+n)} + \beta_{k}\frac{s_{k}}{k^{\alpha}} - x_{k}^{T}\hat{\theta}^{\text{RLS}}_{k+m+n} \\
= & x_{k-(m+n)}^{T}\hat{\theta}_{k-(m+n)} + \beta_{k}\frac{s_{k}}{k^{\alpha}} - x_{k-(m+n)}^{T}\hat{\theta}^{\text{RLS}}_{k} \\
& - \sum_{l=k}^{k+(m+n)-1} a_{l+1}x_{k}^{T}P_{l}\phi_{l+1}\left(y_{l+1}-\phi_{l+1}^{T}\hat{\theta}_{l}^{\text{RLS}}\right) \\
= & x_{k-(m+n)}^{T}\hat{\theta}_{k-(m+n)} + \beta_{k}\frac{s_{k}}{k^{\alpha}} - x_{k-(m+n)}^{T}\hat{\theta}^{\text{RLS}}_{k} \\
& - \sum_{l=k}^{k+(m+n)-1} a_{l+1}x_{k}^{T}P_{l}\phi_{l+1}\left(\phi_{l+1}^{T}\theta+d_{l+1}-\phi_{l+1}^{T}\hat{\theta}_{l}^{\text{RLS}}\right) \\
= & \omega_{k-(m+n)}  + \beta_{k}\frac{s_{k}}{k^{\alpha}}  + \sum_{l=k}^{k+(m+n)-1}a_{l+1} x_{k}^{T}P_{l}\phi_{l+1}\phi_{l+1}^{T}\tilde{\theta}_{l}^{\text{RLS}} \\
& - \sum_{l=k}^{k+(m+n)-1} a_{l+1}x_{k}^{T}P_{l}\phi_{l+1}d_{l+1},
\end{align*}
where $\tilde{\theta}_{k}^{\text{RLS}} \triangleq \hat{\theta}_{k}^{\text{RLS}} - \theta$ is the RLS estimation error.

Define $\eta_{k}  \triangleq  \sum_{l=k}^{k+(m+n)-1}a_{l+1} x_{k}^{T}P_{l}\phi_{l+1}\phi_{l+1}^{T}\tilde{\theta}_{l}^{\text{RLS}}$ and $\zeta_{k} \triangleq - \sum_{l=k}^{k+(m+n)-1} a_{l+1}x_{k}^{T}P_{l}\phi_{l+1}d_{l+1}.$
Then, $\omega_{k}$ can be expressed as
\begin{align}\label{eq1}
\omega_{k} = \omega_{k-(m+n)}  + \beta_{k}\frac{s_{k}}{k^{\alpha}} + \eta_{k} + \zeta_{k}.
\end{align}

Based on \eqref{eq1}, the remainder of the proof is organized into the following four parts.

$\mathbf{Part}$ $\mathbf{1:}$ We will prove 
\begin{align}\label{AAA}
\mathbb{P}\left( \bigcup_{l=k}^{\infty} \left(\mathrm{A}_{l}\cup\mathrm{B}_{l}\right)\right) = O\left(\frac{1}{k^{p}}\right),
\end{align}
where $\mathrm{A}_{k} \triangleq \{ \vert \eta_{k} \vert > \uline{\beta}/(3k^{\alpha})\}$;  $\mathrm{B}_{k} \triangleq \{ \vert \zeta_{k} \vert > \uline{\beta}/(3k^{\alpha})\}$.

Since $\vert a_{l+1} \vert < 1$, we have
\begin{align}\label{sseerr1}
\left\vert \eta_{k} \right\vert \leq \sum_{l=k}^{k+(m+n)-1} \left\Vert P_{l} \right\Vert \left\Vert\phi_{l+1}\right\Vert^{2}\left\Vert\tilde{\theta}_{l}^{\text{RLS}}\right\Vert.
\end{align}
Define $ \mathrm{C}_{l} \triangleq \{ \Vert P_{l} \Vert > (2\Vert (\bar{\mathbb{E}}[ \{ \phi_{k}\phi_{k}^{T}\}_{k=1}^{\infty}])^{-1} \Vert)/l \}$, $\mathrm{D}_{l} \triangleq \{\Vert \phi_{l+1} \Vert  > l^{(1-\alpha)/2} \}$ and $\mathrm{E}_{l} \triangleq \{ \Vert \tilde{\theta}_{l}^{\text{RLS}} \Vert > \uline{\beta}/(6(m+n)\Vert (\bar{\mathbb{E}}[  \{ \phi_{k}\phi_{k}^{T}\}_{k=1}^{\infty}])^{-1} \Vert) \}$.
By Lemma \ref{lemma_b12} and Lemma \ref{lemma_b21} in Appendix \ref{aap_a}, one can get $\mathbb{P}(\mathrm{C}_{l}) = O(1/l^{2p})$ and  $\mathbb{P}(\mathrm{D}_{l}) = O(1/l^{2p})$.
Besides, by Lemma \ref{lemma_aaa444} in Appendix \ref{aap_b} and the Markov inequality \citep{Loeve1977}, it holds that
$\mathbb{P}( \mathrm{E}_{l} ) \leq (6^{4p}(m+n)^{4p}\Vert (\bar{\mathbb{E}}[ \{ \phi_{k}\phi_{k}^{T}\}_{k=1}^{\infty}])^{-1} \Vert^{4p}\mathbb{E}[\Vert \tilde{\theta}_{l}^{\text{RLS}} \Vert^{4p}] )\\/ \uline{\beta}^{4p} = O(1/l^{2p}).$
By (\ref{sseerr1}), DeMorgan's Law \citep{casella2001statistical} and Boole's Inequality \citep{casella2001statistical}, 
\begin{align}\label{asdssd1x}
\mathbb{P}\left( \mathrm{A}_{k} \right) & = 1 -  \mathbb{P}\left( \vert \eta_{k} \vert \leq \frac{\uline{\beta}}{3k^{\alpha}}\right) \nonumber\\
& \leq 1 - \mathbb{P}\left( \bigcap_{l=k}^{k+(m+n)-1} \left( \mathrm{C}_{l}^{\mathrm{c}} \cap \mathrm{D}_{l}^{\mathrm{c}} \cap \mathrm{E}_{l}^{\mathrm{c}} \right) \right)\nonumber\\
& = \mathbb{P}\left( \bigcup_{l=k}^{k+(m+n)-1} \left( \mathrm{C}_{l} \cup \mathrm{D}_{l} \cup \mathrm{E}_{l} \right) \right)\nonumber \\
& = O\left(\frac{1}{k^{2p}}\right).
\end{align}
Since $d_{k} \sim \mathcal{N}(0,\delta_{d}^{2})$, for the positive constant $r_d \triangleq 4p/(1-\alpha)$, $\mathbb{E}[\vert d_{k+1} \vert ^{r_d}] < \infty.$ 
Define $\mathrm{F}_{l} \triangleq \{ \vert d_{l+1} \vert > \uline{\beta}l^{(1-\alpha)/2}/(6(m+n)\Vert (\bar{\mathbb{E}}[\{ \phi_{k}\phi_{k}^{T}\}_{k=1}^{\infty}])^{-1} \Vert) \}$.
Then, by the Markov inequality \citep{Loeve1977}, we have $\mathbb{P}( \mathrm{F}_{l}) \leq (6^{r_d}(m+n)^{r_d}\Vert (\bar{\mathbb{E}}[\{ \phi_{k}\phi_{k}^{T}\}_{k=1}^{\infty}])^{-1} \Vert^{r_d}\mathbb{E}[\vert d_{k+1} \vert^{r_d}] )\\/( \uline{\beta}^{r_d} l^{r_d (1-\alpha)/2})  = O(1/l^{2p}).$

Note that
$\Vert \zeta_{k} \Vert \leq \sum_{l=k}^{k+(m+n)-1} \vert a_{l+1} \vert \Vert x_{k} \Vert \Vert P_{l} \Vert \Vert\phi_{l+1}\Vert \\ \vert d_{l+1}\vert \leq \sum_{l=k}^{k+(m+n)-1} \Vert P_{l} \Vert \Vert\phi_{l+1}\Vert\vert d_{l+1}\vert.$
Then, by DeMorgan's Law \citep{casella2001statistical} and Boole's Inequality \citep{casella2001statistical}, similarly to (\ref{asdssd1x}), one can get
\begin{align}\label{asdssd2x}
\mathbb{P}\left(  \mathrm{B}_{k} \right) & \leq \sum_{l=k}^{k+(m+n)-1} \mathbb{P}\left(\mathrm{C}_{l}\right) + \mathbb{P}\left(\mathrm{D}_{l}\right) + \mathbb{P}\left(\mathrm{F}_{l}\right)\nonumber\\ &= O\left(\frac{1}{k^{2p}}\right).
\end{align}
Hence, by (\ref{asdssd1x}) and (\ref{asdssd2x}), we have $\mathbb{P}( \cup_{l=k}^{\infty} (\mathrm{A}_{l}\cup\mathrm{B}_{l})) = O(\sum_{l=k}^{\infty}1/l^{2p}) = O(1/k^{2p-1}).$
Since $p$ is a positive integer, we have $2p-1 \geq p$, which indicates (\ref{AAA}) holds.

$\mathbf{Part}$ $\mathbf{2:}$ We will prove
\begin{align}\label{CCC}
\mathbb{P}\left( \left\vert \omega_{k} \right\vert > 2c_{x}k^{1-\alpha} \right) = O\left(\frac{1}{k^{p}}\right),
\end{align}
where $c_{x} \triangleq \Vert \hat{\theta}_{0} \Vert + \bar{\beta} + \bar{\beta}/(1-\alpha)$.

Since $\Vert x_{k} \Vert = 1$, $\vert s_{k} \vert = 1$ hold uniformly, by $\hat{\theta}_{k} = \hat{\theta}_{k-1} + \beta_{k}x_{k}s_{k}/k^{\alpha}$, we have
$\Vert \hat{\theta}_{k} - \hat{\theta}_{k-1} \Vert \leq \beta_{k}/k^{\alpha} \leq \bar{\beta}/k^{\alpha}$ holds uniformly.
Thus, 
\begin{align*}
\left\Vert \hat{\theta}_{k} \right\Vert & \leq \Vert \hat{\theta}_{0} \Vert + \bar{\beta}\sum_{l=1}^{k} \frac{1}{l^{\alpha}}  \leq \Vert \hat{\theta}_{0} \Vert + \bar{\beta}\left( 1+ \int_{1}^{k} \frac{1}{x^\alpha}dx \right) \\
& = \Vert \hat{\theta}_{0} \Vert + \bar{\beta}\left( 1+ \frac{k^{1-\alpha}-1}{1-\alpha} \right) \leq  c_{x} k^{1-\alpha}.
\end{align*}
Then, by $\Vert x_{k} \Vert = 1$, one can get
\begin{align}\label{uuuiii}
\left\vert x_{k}^{T}\hat{\theta}_{k} \right\vert \leq \left\Vert x_{k} \right\Vert \left\Vert \hat{\theta}_{k} \right\Vert \leq c_{x}k^{1-\alpha}.
\end{align}
Hence, by $\omega_{k} = x_{k}^{T}\hat{\theta}_{k} - x_{k}^{T}\hat{\theta}^{\text{RLS}}_{k+m+n}$, one can get
\begin{align}\label{tt55uu}
\mathbb{P}\left(\left\vert \omega_{k} \right\vert > 2c_{x}k^{1-\alpha}\right) & \leq \mathbb{P}\left(\left\vert x_{k}^{T}\hat{\theta}^{\text{RLS}}_{k+m+n} \right\vert > c_{x}k^{1-\alpha}\right) \nonumber \\
& \leq \mathbb{P}\left(\left\Vert \hat{\theta}^{\text{RLS}}_{k+m+n} \right\Vert > c_{x}k^{1-\alpha}\right).
\end{align}
Since $\hat{\theta}_{k+m+n}^{\text{RLS}} = \tilde{\theta}_{k+m+n}^{\text{RLS}} + \theta$, and for sufficiently large $k$, $\Vert \theta \Vert < c_{x}k^{1-\alpha}/2$, by DeMorgan's Law \citep{casella2001statistical} and Boole's Inequality \citep{casella2001statistical}, we have $\mathbb{P}(\Vert \hat{\theta}^{\text{RLS}}_{k+m+n} \Vert > c_{x}k^{1-\alpha}) 
\leq \mathbb{P}(\Vert \tilde{\theta}^{\text{RLS}}_{k+m+n} \Vert > c_{x}k^{1-\alpha}/2).$
Then, by Lemma \ref{lemma_aaa444} in Appendix \ref{aap_b} and the Markov inequality \citep{Loeve1977}, $ \mathbb{P}( \Vert \tilde{\theta}_{k+m+n}^{\text{RLS}}  \Vert > c_{x}k^{1-\alpha}/2) \leq  (2^{\lceil p/(1-\alpha) \rceil}\mathbb{E}[\Vert \tilde{\theta}_{k+m+n}^{\text{RLS}} \Vert^{\lceil p/(1-\alpha) \rceil}]) \\ /(c_{x}^{\lceil p/(1-\alpha) \rceil} k^{\lceil p/(1-\alpha) \rceil(1-\alpha)}) = O(1/k^{p}).$
Besides, when $k$ is sufficiently large, $\Vert \theta \Vert < c_{x}k^{1-\alpha}/2$, which implies that $\mathbb{P}(\Vert \hat{\theta}^{\text{RLS}}_{k+m+n} \Vert > c_{x}k^{1-\alpha}) = O(1/k^{p})$.
Then, by (\ref{tt55uu}), one can get (\ref{CCC}) holds.

$\mathbf{Part}$ $\mathbf{3:}$ 
We will prove that, for sufficiently large $k$ and $t = \lfloor k/c_H \rfloor + \text{mod}(k- \lfloor k/c_H\rfloor,m+n) - (m+n)$, 
\begin{align}\label{incur}
\left\vert \omega_{k} \right\vert \leq \frac{\rho}{k^{\alpha}}, \quad \text{on} \quad \mathrm{G}_{t}^{\mathrm{c}},
\end{align}
where $c_{H} \triangleq 1+(m+n)c_{T}$, $c_{T} \triangleq ( (2(m+n)(1-\alpha)c_{x})/\varrho + 2^{1-\alpha} )^{1/(1-\alpha)}$ and $\varrho \triangleq \uline{\beta}/3$, and $\mathrm{G}_{t} \triangleq (\cup_{l=t}^{\infty}(\mathrm{A}_{l}\cup\mathrm{B}_{l}))\cup \{ \vert \omega_{t} \vert > 2c_{x}t^{1-\alpha} \}$.

The analysis in this part is carried out on the set $\mathrm{G}_{t}^{\mathrm{c}}$. 
For simplicity of description, we will not emphasize this point in this part. 

By (\ref{sss222333iii}), one can get $s_{k}
= I_{\{x_{k}^{T}\hat{\theta}_{k}^{\text{RLS}} - x_{k}^{T}\hat{\theta}_{k-(m+n)}>0\}}  -  I_{\{x_{k}^{T}\hat{\theta}_{k}^{\text{RLS}} - x_{k}^{T}\hat{\theta}_{k-(m+n)} \leq 0\}}$.
Thus, 
$
s_{k} = I_{\{ \omega_{k-(m+n)} < 0\}}-I_{\{ \omega_{k-(m+n)} \geq 0\}}.
$
Then, for $l \geq t$:

Since $\vert \eta_{l} \vert \leq \uline{\beta}/(3l^{\alpha})$ and $\vert \zeta_{l} \vert \leq \uline{\beta}/(3l^{\alpha})$, by (\ref{eq1}), it follows that
 \begin{align}
\label{tt3} \frac{\varrho}{l^{\alpha}} \leq \left\vert \omega_{l} - \omega_{l-(m+n)} \right\vert \leq \frac{\rho}{l^{\alpha}}.
\end{align}
Since $\vert \eta_{l} \vert + \vert \zeta_{l} \vert 
< \vert \beta_{l}s_{l}/l^\alpha \vert$, by (\ref{eq1}), the sign of $\omega_{l} - \omega_{l-(m+n)}$ is the same as that of $s_l$.
Moreover, since the sign of $\omega_{l-(m+n)}$ is opposite to that of $s_l = I_{\{ \omega_{l-(m+n)} < 0\}}-I_{\{ \omega_{l-(m+n)} \geq 0\}}$, we obtain 
\begin{align}\label{tt31}
 (\omega_{l} - \omega_{l-(m+n)})\omega_{l-(m+n)} \leq 0.
\end{align}
Besides, by (\ref{tt3}), when $\vert \omega_{l-(m+n)} \vert$ exceeds the maximal possible correction step $\rho/l^{\alpha}$, the sign of $\omega_{l-(m+n)}$ is the same as that of $\omega_{l}$.
In this case, since $(\omega_{l} - \omega_{l-(m+n)})\omega_{l-(m+n)} \leq 0$, one can get $\vert \omega_{l}\vert < \vert \omega_{l-(m+n)} \vert$.
Furthermore, since the correction step admits a positive lower bound $\varrho/l^{\alpha}$, one can get 
\begin{align}\label{tt23}
\left\vert \omega_{l} \right\vert < \left\vert \omega_{l-(m+n)} \right\vert - \frac{\varrho}{l^{\alpha}}.
\end{align}
By (\ref{tt23}), we claim that there exists a nonnegative integer $T(t)$ such that
\begin{align}
\left\vert \omega_{t+(m+n)T(t)} \right\vert \leq \frac{\rho}{(t+(m+n)T(t))^{\alpha}}.
\end{align}
Otherwise, if $\vert \omega_{t+(m+n)l} \vert > \rho/(t+(m+n)l)^{\alpha}$ for all $l\ge0$, by (\ref{tt23}), for a large integer $L$, we have 
$
\vert \omega_{t+(m+n)L} \vert
\le \vert\omega_t \vert -\sum_{l=1}^{L} \varrho/(t+(m+n)l)^{\alpha}.
$
Since $1/2<\alpha<1$ implies $\sum_{l=1}^{\infty}(t+(m+n)l)^{-\alpha}=\infty$, $\vert \omega_{t+(m+n)L} \vert < 0$ for a sufficient large $L$, which concludes a contradiction.

Define $T^{*}(t) \triangleq \min\{l : \vert \omega_{t+(m+n)l} \vert \leq\rho/(t+(m+n)l)^{\alpha}\}$. 
By (\ref{tt23}), we have 
\begin{align}\label{neweq1}
& \left\vert \omega_{t} \right\vert - \left\vert \omega_{t + (m+n)T^{*}(t)} \right\vert 
> \sum_{l=1}^{T^{*}(t)} \frac{\varrho}{(t+(m+n)l)^{\alpha}} \nonumber \\
> & \int_{2}^{T^{*}(t) +1}\frac{\varrho }{(t+(m+n)x)^{\alpha}} dx \nonumber \\
= & \frac{\varrho((t+(m+n)(T^{*}(t) +1))^{1-\alpha}-(t+(m+n)2)^{1-\alpha}) }{(m+n)(1-\alpha)}.
\end{align}
Besides, by $\vert \omega_{t} \vert \leq 2c_{x}t^{1-\alpha}$, it holds that
\begin{align}\label{neweq2}
\left\vert \omega_{t} \right\vert - \left\vert \omega_{t + (m+n)T^{*}(t) }  \right\vert \leq 2c_{x}t^{1-\alpha}.
\end{align}
Thus, by (\ref{neweq1}) and (\ref{neweq2}), we have 
\begin{align*}
2c_{x}t^{1-\alpha} \geq & \frac{\varrho(t+(m+n)(T^{*}(t) +1))^{1-\alpha}}{(m+n)(1-\alpha)} \\ & - \frac{\varrho(t+(m+n)2)^{1-\alpha}}{(m+n)(1-\alpha)}.
\end{align*}
Since $\alpha < 1$, $\varrho/((m+n)(1-\alpha)) > 0$. It follows that 
\begin{align*}
& \frac{2(m+n)(1-\alpha)c_{x}t^{1-\alpha}}{ \varrho} \\ \geq &(t+(m+n)(T^{*}(t) +1))^{1-\alpha} - (t+(m+n)2)^{1-\alpha} \\
> & (T^{*}(t) )^{1-\alpha} - (t+(m+n)2)^{1-\alpha}.
\end{align*}
Hence, $T^{*}(t)  < ( 2(1-\alpha)(m+n)c_{x}t^{1-\alpha}/\varrho + (t+(m+n)2)^{1-\alpha})^{1/(1-\alpha)} $.
Since $k$ is sufficiently large, $t \geq 2(m+n)$.
Thus, one can get $(t+(m+n)2)^{1-\alpha} \leq (2t)^{1-\alpha} = 2^{1-\alpha}t^{1-\alpha}$, which indicates that
\begin{align}
T^{*}(t) & <  \left( \frac{2(m+n)(1-\alpha)c_{x}t^{1-\alpha}}{\varrho} + 2^{1-\alpha} t^{1-\alpha}\right)^{1/(1-\alpha)} \nonumber \\
& = \left( \left(\frac{2(m+n)(1-\alpha)c_{x}}{\varrho} + 2^{1-\alpha}\right) t^{1-\alpha}\right)^{1/(1-\alpha)} \nonumber \\
& = \left( \frac{2(m+n)(1-\alpha)c_{x}}{\varrho} + 2^{1-\alpha}\right)^{1/(1-\alpha)}t = c_{T}t.
\end{align} 
In addition, by (\ref{tt3}) and (\ref{tt31}), for the time index $T^{*}(t) +1$, we have 
\begin{align*}
& \left\vert \omega_{t+(m+n)(T^{*}(t) +1)} \right\vert \nonumber \\
\leq &\max\bigg\{ \frac{\rho}{(t+(m+n)(T^{*}(t) +1))^{\alpha}},\nonumber \\
& \frac{\rho}{(t+(m+n)T^{*}(t) )^{\alpha}} - \frac{\varrho}{(t+(m+n)(T^{*}(t) +1))^{\alpha}} \bigg\}.
\end{align*}
Since $k$ is sufficiently large, $t \geq (m+n)/((1+\varrho/\rho)^{1/\alpha}-1)$.
Thus, one can get $1/((1+\varrho/\rho)^{1/\alpha}-1) - t/(m+n) \leq 0$, which implies that 
$T^{*}(t)  > 1/((1+\varrho/\rho)^{1/\alpha}-1) - t/(m+n)$.
Thus,
\[
t + \left(m+n\right)T^{*}(t) > \frac{m+n}{(1+\varrho/\rho)^{1/\alpha}-1}.
\]
Since $\varrho/\rho > 0$, we have $((1+\varrho/\rho)^{1/\alpha}-1) (t + (m+n)T^{*}(t)) > m+n, $
which follows that 
\[
1+\frac{\varrho}{\rho}  > \left( \frac{t + (m+n)(T^{*}(t)+1)}{t + (m+n)T^{*}(t)} \right)^{\alpha}.
\]
Thus, we have
$ \rho/(t+(m+n)(T^{*}(t)+1))^{\alpha}  >  \rho/(t+(m+n)T^{*}(t))^{\alpha} - \varrho/(t+(m+n)(T^{*}(t)+1))^{\alpha}$, which implies 
\[
\left\vert \omega_{t+(m+n)(T^{*}(t) +1)} \right\vert \leq \frac{\rho}{(t+(m+n)(T^{*}(t)+1))^{\alpha}}.
\]
Repeating the above analysis, we have
\[
\left\vert \omega_{t+(m+n)l} \right\vert \leq \frac{\rho}{(t+(m+n)l)^{\alpha}}, \quad l \geq T^{*}(t).
\]
Since $T^{*}(t)  \leq c_{T}t$ and $t \leq  k/c_{H} = k/(c_{T}(m+n)+1)$, we have $k \geq t + c_{T}(m+n)t \geq t + (m+n)T^{*}(t)$.
Since $k$ can be written as $k = \lfloor k/c_H \rfloor + \mod(k - \lfloor k/c_H \rfloor, m+n) - (m+n) + q(m+n) = t + q(m+n)$, where $q$ is an integer, one can get $\mod(k,m+n) = \mod(t,m+n)$.
Thus on $\mathrm{G}_{t}^{\mathrm{c}}$, we have (\ref{incur}) holds.

$\mathbf{Part}$ $\mathbf{4:}$ Prove the conclusion of Theorem \ref{theorem1}.

From (\ref{AAA}) in Part 1 and (\ref{CCC}) in Part 2, one can get
\[
\mathbb{P}\left( \mathrm{G}_{t}\right) = O\left(\frac{1}{t^{p}}\right).
\]
From (\ref{incur}) in Part 3 and $t \geq \lfloor k/c_H \rfloor - (m+n)$, we have
\begin{align*}
\mathbb{P}\left(\left\vert \omega_{k} \right\vert > \frac{\rho}{k^{\alpha}} \right) & \leq \mathbb{P}\left( \mathrm{G}_{t}\right) = O\left(\frac{1}{t^{p}}\right) \\
& \leq O\left(\frac{1}{ \lfloor k/c_{H}\rfloor^{p}}\right) = O\left(\frac{1}{k^{p}}\right),
\end{align*}
which indicates that \eqref{main} holds. \hfill $\qed$
\end{pf}

\begin{rem}
Theorem \ref{theorem1} derives the tail probability of the difference between $\hat{\theta}_k$ and $\hat{\theta}_k^{\text{RLS}}$, which converges to zero at an arbitrary polynomial rate.
This result provides two benefits:

\romannumeral1) It provides a framework for analyzing the convergence properties of $\hat{\theta}_k$ based on $\hat{\theta}_k^{\text{RLS}}$ tailored to the non-independent $s_{k}$ under quasi-stationary $u_{k}$.  
For instance, Theorem \ref{theorem1} together with Borel-Cantelli lemma \citep{ash2014real} implies the almost sure convergence of $\omega_k$. 
Then, the almost sure convergence of $\hat{\theta}_k$ can be further achieved by the almost sure convergence of $\hat{\theta}_k^{\text{RLS}}$. 

\romannumeral2) According to Lemma \ref{lemma_aaa444}, Theorem \ref{theorem1} also implies that the tail probability of $\hat{\theta}_k$ converges to zero at an arbitrary polynomial rate.
Due to the small tail probability, the case of the unbounded remote estimate $\hat{\theta}_k$ is negligible. 
Therefore, Theorem \ref{theorem1} can replace the role of the uniform boundedness guaranteed by the prior assumption in \citet{zhang2019asymptotically, wang2023identification, wang2024asymptotically} that the system parameter resides within a known compact set for their convergence analysis. 
\end{rem}

\subsection{Almost sure and $L^{p}$ convergence of the RLS-SA algorithm}\label{sec_c2}

This subsection establishes both almost sure convergence and \( L^{p} \) convergence for any positive integer \( p \), demonstrating that the RLS-SA algorithm is capable of effectively addressing the system identification task.

\begin{thm}\label{theorem2}
Under Assumptions \ref{ass_a}-\ref{ass_c}, $\hat{\theta}_{k}$ given by the RLS-SA algorithm converges to $\theta$ in the almost sure sense with a convergence rate of $O(\sqrt{\log k/k})$, i.e., 
\begin{align}
\label{theorem_1} \left\Vert \tilde{\theta}_{k} \right\Vert = O\left(\sqrt{\frac{\log k}{k}}\right) \quad  \mathrm{a.s.}
\end{align}
where $\tilde{\theta}_{k} \triangleq \hat{\theta}_{k} - \theta$ is the estimation error.
\end{thm}

\begin{pf}
By (\ref{main}), we have $\sum_{k=1}^{\infty} \mathbb{P}(\vert \omega_{k} \vert > \rho/k^{\alpha}) < \infty$. 
Then, by Borel-Cantelli lemma \citep{ash2014real}, it holds that
\[
\left\vert \omega_{k} \right\vert = O\left(\frac{1}{k^{\alpha}}\right) \quad  \mathrm{a.s.}
\]
For $i = 0,1,2,\ldots,m+n-1$, one can get $x_{k-i}^{T}\hat{\theta}_{k} = x_{k-i}^{T}\hat{\theta}_{k-i}$.
It holds that
\begin{align*}
& \left\vert x_{k-i}^{T}\hat{\theta}_{k} - x_{k-i}^{T}\hat{\theta}_{k+m+n-i}^{\text{RLS}}\right\vert \nonumber\\
= & \left\vert \omega_{k-i}  \right\vert  = O\left(\frac{1}{(k-i)^{\alpha}}\right) =  O\left(\frac{1}{k^{\alpha}}\right) \quad  \mathrm{a.s.}
\end{align*}
Besides, by Lemma \ref{lemma_a111} in Appendix \ref{aap_b}, we have
\begin{align*}
\left\Vert \hat{\theta}_{k+m+n}^{\text{RLS}} - \hat{\theta}_{k+m+n-i}^{\text{RLS}} \right\Vert & \leq \sum_{j=1}^{i}\left\Vert \hat{\theta}_{k+m+n-j+1}^{\text{RLS}} - \hat{\theta}_{k+m+n-j}^{\text{RLS}} \right\Vert \\
& = O\left(\frac{i}{k^{\alpha}}\right) = O\left(\frac{1}{k^{\alpha}}\right)  \quad \mathrm{a.s.} 
\end{align*}
Hence, 
\begin{align*}
\left\vert x_{k-i}^{T}\hat{\theta}_{k} - x_{k-i}^{T}\hat{\theta}_{k+m+n}^{\text{RLS}}\right\vert
= O\left(\frac{1}{k^{\alpha}}\right) \quad \mathrm{a.s.} 
\end{align*}
Note that $x_{k-i}^{T}\tilde{\theta}_{k}$ denotes the value of the $\mathrm{mod}(k-i,m+n)+1$-th dimension of $\tilde{\theta}_{k}$.
Thus,
\begin{align}\label{thetaas}
\left\Vert \hat{\theta}_{k} - \hat{\theta}_{k+m+n}^{\text{RLS}}\right\Vert =  O\left(\frac{1}{k^{\alpha}}\right) \quad \mathrm{a.s.} 
\end{align}
Then, by Lemma \ref{lemma_c} in Appendix \ref{aap_b} and $\Vert \tilde{\theta}_{k} \Vert  \leq \Vert \hat{\theta}_{k} - \hat{\theta}_{k+m+n}^{\text{RLS}} \Vert + \Vert \tilde{\theta}_{k+m+n}^{\text{RLS}}\Vert $, one can get (\ref{theorem_1}) holds. {\hfill $\qed$}
\end{pf}

\begin{thm}\label{theorem2b}
Under Assumptions \ref{ass_a}-\ref{ass_c}, $\hat{\theta}_{k}$ given by the RLS-SA algorithm converges to $\theta$ in the $L^{p}$ sense with a convergence rate of $O(1/k^{p/2})$, i.e.,
\begin{align}
\label{theorem_2a} \mathbb{E}\left[ \left\Vert \tilde{\theta}_{k} \right\Vert^{p} \right] = O\left(\frac{1}{k^{p/2}}\right),
\end{align}
where $p$ is an arbitrary positive integer.
\end{thm}

\begin{pf}
Since $I_{\{ \vert \omega_{k} \vert > \rho/k^{\alpha} \}}^{2} = I_{\{ \vert \omega_{k} \vert > \rho/k^{\alpha} \}}$ and $\omega_{k} = x_{k}^{T}\hat{\theta}_{k} - x_{k}^{T}\tilde{\theta}^{\text{RLS}}_{k+m+n} - x_{k}^{T}\theta,$ by the $C_{r}$-inequality \citep{Loeve1977}, we have $\mathbb{E}[\vert\omega_{k}\vert^{p}] = 
\mathbb{E}[\vert\omega_{k}\vert^{p} I_{\{ \vert \omega_{k} \vert > \rho/k^{\alpha} \}}]
+ \mathbb{E}[\vert\omega_{k}\vert^{p} I_{\{ \vert \omega_{k} \vert \leq \rho/k^{\alpha} \}}] 
\leq 2^{p-1}\mathbb{E}[\vert x_{k}^{T}\hat{\theta}_{k} \vert^{p} I_{\{ \vert \omega_{k} \vert > \rho/k^{\alpha} \}}] 
+ 2^{p-1}\mathbb{E}[\Vert \tilde{\theta}^{\text{RLS}}_{k+m+n} \Vert^{p} I_{\{ \vert \omega_{k} \vert > \rho/k^{\alpha} \}}]  
+ 2^{p-1}\mathbb{E}[\Vert \theta \Vert^{p} I_{\{ \vert \omega_{k} \vert > \rho/k^{\alpha} \}}] 
+ \mathbb{E}[\vert\omega_{k}\vert^{p} I_{\{ \vert \omega_{k} \vert \leq \rho/k^{\alpha} \}}].
$

By (\ref{main}) and (\ref{uuuiii}), one can get $\mathbb{E}[\vert x_{k}^{T}\hat{\theta}_{k} \vert^{p} I_{\{ \vert \omega_{k} \vert > \rho/k^{\alpha} \}}] =  O(1/k^{p\alpha}).$
Besides, by Lemma \ref{lemma_aaa444} in Appendix \ref{aap_b}, Theorem \ref{theorem1}, and H\"{o}lder's Inequality \citep{Loeve1977}, it holds that $ \mathbb{E}[\Vert \tilde{\theta}^{\text{RLS}}_{k+m+n} \Vert^{p} I_{\{ \vert \omega_{k} \vert > \rho/k^{\alpha} \}}] \\\leq (\mathbb{E}[\Vert \tilde{\theta}^{\text{RLS}}_{k+m+n} \Vert^{2p}] \mathbb{E}[I_{\{ \vert \omega_{k} \vert > \rho/k^{\alpha} \}}^{2}])^{1/2} = O(1/k^{p}).$
In addition, by Theorem \ref{theorem1}, we have $\mathbb{E}[\Vert \theta \Vert^{2} I_{\{ \vert \omega_{k} \vert > \rho/k^{\alpha} \}}] 
= O(1/k^{p}).$
Therefore, together with $\mathbb{E}\left[\vert\omega_{k}\vert^{p}I_{\{ \vert \omega_{k} \vert \leq \rho/k^{\alpha} \}}\right] = O(1/k^{p\alpha})$, one can get $\mathbb{E}\left[\vert\omega_{k}\vert^{p}\right] = O(1/k^{p\alpha}).$ 

Thus, for $i = 0,1,2,\ldots,m+n-1$, we obtain $ \mathbb{E}[\vert x_{k-i}^{T}\hat{ \theta}_{k} - x_{k-i}^{T}\hat{\theta}_{k+m+n-i}^{\text{RLS}}\vert^{p} ]  = \mathbb{E} [ \vert \omega_{k-i} \vert^{p}] = O(1/(k-i)^{p\alpha}) = O(1/k^{p\alpha}).$
Then, by Lemma \ref{lemma_aaa444} in Appendix \ref{aap_b} and the $C_{r}$-inequality \citep{Loeve1977}, for $i = 0,1,2,\ldots,m+n-1$, it holds that $\mathbb{E}[\vert x_{k-i}^{T}\tilde{\theta}_{k}\vert^{p}] \leq  2^{p-1}\mathbb{E}[\vert x_{k-i}^{T}\hat{\theta}_{k-i} - x_{k-i}^{T}\hat{\theta}_{k+m+n-i}^{\text{RLS}}\vert^{p}] + 2^{p-1}\mathbb{E}[\Vert \tilde{\theta}_{k+m+n-i}^{\text{RLS}} \Vert^{p}] =  O(1/k^{p/2}),$
which implies that (\ref{theorem_2a}) holds. {\hfill $\qed$}

\end{pf}

\begin{rem}
Studying convergence in \( L^{p} \) for different values of \( p \) provides a refined approach for estimating errors.  
For instance, $L^{1}$ convergence ensures that the remote estimate is asymptotically unbiased; $L^{2}$ convergence corresponds to mean square convergence; $L^{4}$ is particularly useful for enhancing robustness against outliers \citep{Zhai2020}.
\end{rem} 

\begin{rem}
The $L^{p}$ convergence rate of the RLS estimator established in \citet{Liu2025} provides a theoretical foundation for Theorem~\ref{theorem2b}. 
By showing that the $L^p$ convergence rate of $\omega_k$ is faster than that of $\hat{\theta}_{k}^{\mathrm{RLS}}$, Theorem~\ref{theorem2b} establishes the $L^p$ convergence rate of $\hat{\theta}_{k}$.
\end{rem} 

\subsection{The asymptotic efficiency of the RLS-SA algorithm}\label{sec_c3}

Finally, the asymptotic efficiency of the RLS-SA algorithm is presented as the following theorem.

\begin{thm}\label{theorem2a}
Under Assumptions \ref{ass_a}-\ref{ass_c}, $\hat{\theta}_{k}$ given by the RLS-SA algorithm is an asymptotically efficient estimate of $\theta$ in both distribution and covariance senses:

\romannumeral1) Asymptotic efficiency in the distribution sense \citep{Fabian1978}:
\begin{align}
\label{theorem_2} \sqrt{k}\tilde{\theta}_{k} \xrightarrow{d}  \mathcal{N}\left(0, \bar{\Sigma}_{\mathrm{CR}} \right),
\end{align}
where 
$\bar{\Sigma}_{\mathrm{CR}} = \lim\limits_{k \to \infty}k\Sigma_{\mathrm{CR}}(k) = \delta_{d}^{2}(\bar{\mathbb{E}}[ \{ \phi_{k}\phi_{k}^{T}\}_{k=1}^{\infty}])^{-1},
$ and 
$\Sigma_{\mathrm{CR}}(k) = \delta_{d}^{2}(\sum_{l=1}^{k}\mathbb{E}[\phi_{l}\phi_{l}^{T}])^{-1}$ is the original CRLB.

\romannumeral2) Asymptotic efficiency in the covariance sense \citep{Rao1961}:
\begin{align}
\label{theorem_3}\lim\limits_{k \rightarrow \infty} k\mathbb{E}\left[\tilde{\theta}_{k}\tilde{\theta}_{k}^{T} \right] = \bar{\Sigma}_{\mathrm{CR}}.
\end{align}
\end{thm}

\begin{pf}
By (\ref{thetaas}), one can get
\begin{align*}
\lim\limits_{k \rightarrow \infty} \left\Vert \sqrt{k}\hat{\theta}_{k} - \sqrt{k}\hat{\theta}_{k+m+n}^{\text{RLS}}\right\Vert = 0 \quad  \mathrm{a.s.}
\end{align*}
Then, by Lemma \ref{lemma_aaa444} in Appendix \ref{aap_b}, we have (\ref{theorem_2}) holds. 
Note that $\mathbb{E}[\Vert \sqrt{k}\tilde{\theta}_{k} \Vert^{4}]  < \infty.$
Then, by  Lemma \ref{lemma_aaa222} in Appendix \ref{aap_a}, it holds that
\begin{align*}
\lim\limits_{k\rightarrow\infty} k \mathbb{E}\left[\tilde{\theta}_{k}\tilde{\theta}_{k}^{T} \right] = \delta_{d}^{2}\left(\bar{\mathbb{E}}\left[ \{ \phi_{k}\phi_{k}^{T}\}_{k=1}^{\infty}\right] \right)^{-1},
\end{align*}
which implies that (\ref{theorem_3}) holds. {\hfill $\qed$}
\end{pf}

\begin{rem}
From an encoding perspective, the proposed RLS-SA algorithm can be interpreted as performing semantic information extraction \citep{Lu2022} to improve estimation accuracy.
Specifically, instead of directly quantizing raw measurement data, the encoder first extracts task-relevant parameter information through the RLS algorithm, which provides richer parameter information in one-bit data.
\end{rem}

\section{Numerical examples}\label{sec_e}
This section presents the theoretical results through simulation experiments. 
In addition, we compare the MSE achieved by the RLS-SA algorithm with that of three other types of algorithms:

\romannumeral1) The recursive identification algorithm for ARX systems under one-bit communications \citep{csaji2012recursive}.

\romannumeral2) The asymptotically efficient recursive identification algorithm using a fixed quantizer under one-bit communications \citep{yin2007asymptotically, wang2024asymptotically}.

\romannumeral3) The asymptotically efficient recursive identification algorithm using an adaptive quantizer under one-bit communications \citep{you2015recursive}.

\begin{rem}
To the best of our knowledge, the algorithms selected for comparison are among the best-performing methods currently available in the literature on one-bit system identification \citep{yin2007asymptotically, wang2024asymptotically, you2015recursive}. Therefore, we adopted them as benchmarks to evaluate the performance of the proposed method.
\end{rem}

Since the proposed RLS-SA algorithm is more general in the system model and can adapt to a wider range of system input types, we first designed the following simulation experiments to verify its performance advantages.

\subsection{Convergence properties of the RLS-SA algorithm}

We consider an ARX system as:
\begin{align*}
y_{k} & + a_{1}y_{k-1} + a_{2}y_{k-2} \\
& =  b_{1}u_{k-1} + b_{2}u_{k-2} + b_{3}u_{k-3} + d_{k},\quad k \geq 1,
\end{align*}
where the system input $u_{k}$ is generated from a bounded deterministic sequence and a stochastic sequence: $u_{k} = 3\sin (\pi k/4) + \mathcal{N}(0,4)$; the unknown parameter $\theta = [a_{1},a_{2},b_{1},b_{2},b_{3}]^{T}=[0.2,1,0.6,-0.2,-0.6]^{T}$; the observation noise $d_{k}$ follows $\mathcal{N}(0,0.25)$.

We apply the RLS-SA algorithm to estimate the parameter \(\theta\).
The initial values can be selected based on available prior information, such as the prior distribution of the unknown parameters.
In this simulation,  we assume that $\theta$ follows an unbiased prior distribution.
Accordingly, the initial estimates are chosen as
\(\hat{\theta}^{\mathrm{RLS}}_{0} = [0,0,0,0,0]^{T}\) and \(\hat{\theta}_{0} = [0,0,0,0,0]^{T}\), with the initial covariance matrix \(P_{0} = I_{5}/10\).
Besides, as established in Theorem~\ref{theorem1}, the absolute difference in each corresponding component between the remote estimate $\hat{\theta}_{k}$ and the local RLS estimate $\hat{\theta}_{k}^{\mathrm{RLS}}$ (i.e., $\omega_{k}$) decreases as the step-size coefficient $\alpha$ approaches~1.
Accordingly, in the simulations, the step-size coefficients are set to \(\alpha = 0.95\) and \(\beta_{k} \equiv 1\).

Fig. \ref{fig_a} shows that all estimates given by the RLS-SA algorithm successfully converge to the true value.

\begin{figure}[h]
\noindent\includegraphics[width=0.48\textwidth]{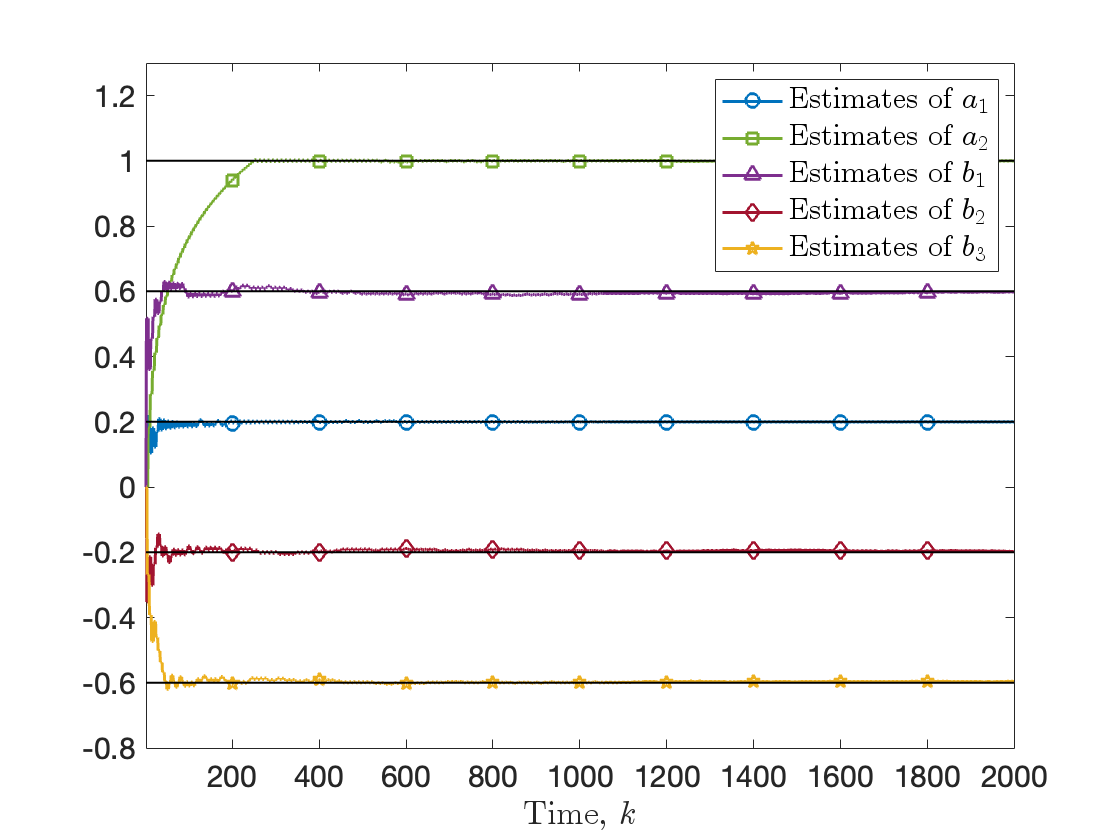}
\caption{Trajectories of estimates given by the RLS-SA algorithm}
\label{fig_a}
\end{figure}

Furthermore, Fig. \ref{fig_d} demonstrates that the trajectory of \( \sqrt{k/\log k} \Vert \tilde{\theta}_{k} \Vert \) remains bounded, confirming that the almost sure convergence rate of \( \hat{\theta}_{k} \) of the RLS-SA algorithm is \( O(\sqrt{\log k/k}) \). 
Due to the one-bit transmission constraint, the remote estimate initially struggles to track the parametric knowledge extracted by the RLS algorithm, resulting in a relatively large estimation error compared with the RLS algorithm. 
As time progresses,  the proposed algorithm enables the remote estimate to gradually converge to the RLS estimate, ultimately achieving the same convergence properties.

\begin{figure}[h]
\noindent\includegraphics[width=0.48\textwidth]{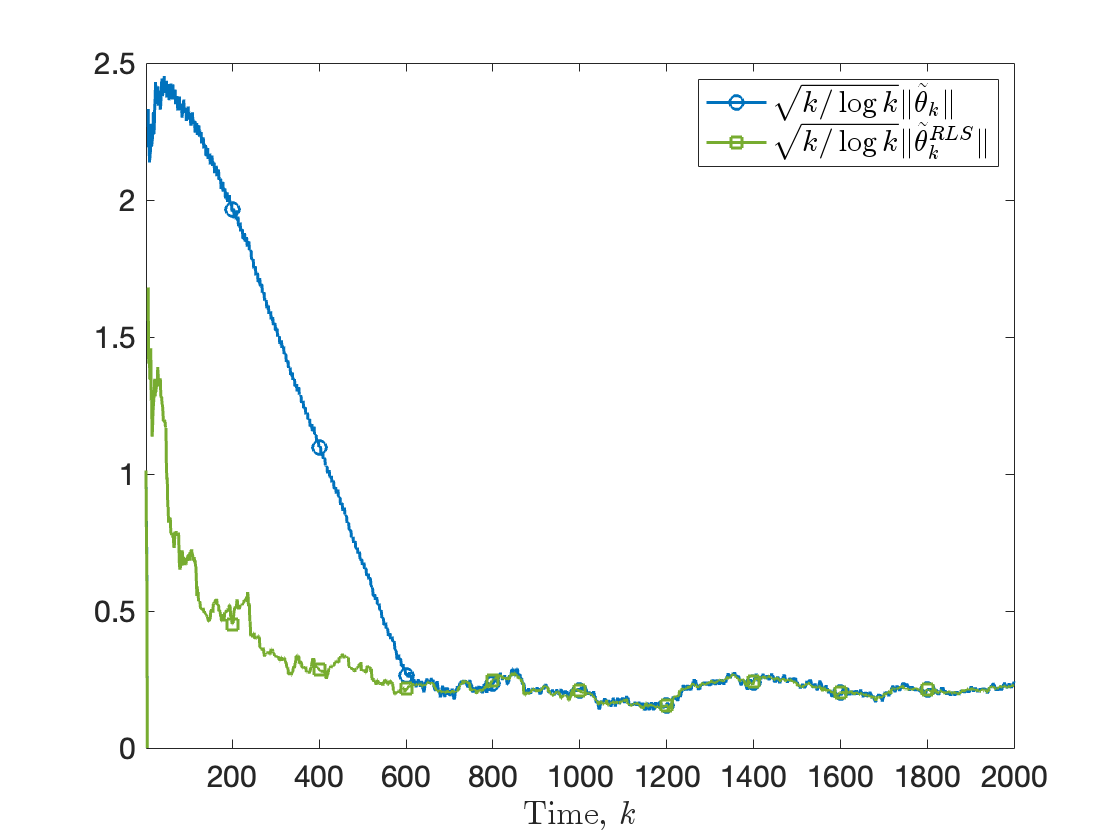}
\caption{Almost sure convergence of the RLS-SA algorithm}
\label{fig_d}
\end{figure}

\begin{figure}[h]
\noindent\includegraphics[width=0.48\textwidth]{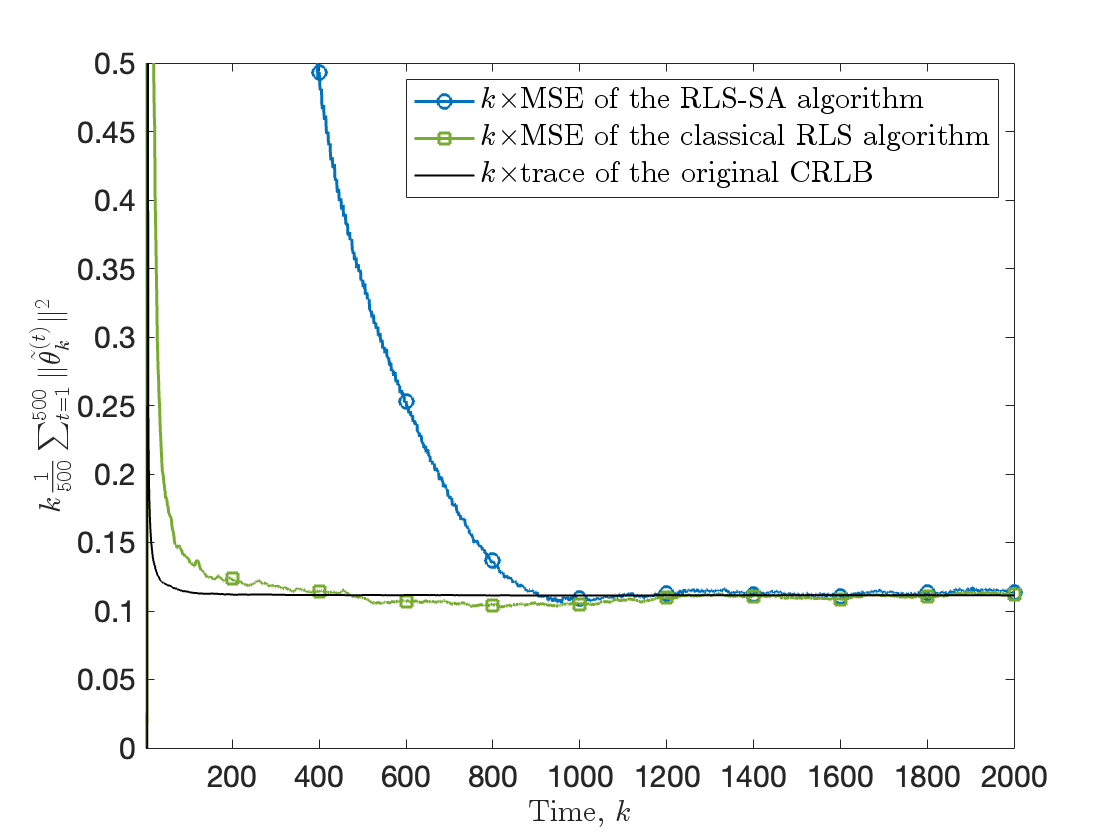}
\caption{Asymptotic efficiency of the RLS-SA algorithm}
\label{fig_d1}
\end{figure}

Additionally, we conducted 500 simulations to calculate the MSE of the RLS-SA algorithm and compared the results with those from the RLS algorithm.
Specifically, we denote the parameter estimation error in the \( t \)-th simulation at time $k$ as \( \tilde{\theta}_{k}^{(t)} \), and use \( \sum_{t=1}^{500} \Vert \tilde{\theta}_{k}^{(t)} \Vert^2 /500\) as the empirical MSE at time $k$.  
Figure \ref{fig_d1} demonstrates that after 1000 steps, the $k$ times MSE of the proposed RLS-SA algorithm, which relies on one-bit information, aligns with that of the RLS algorithm, thereby asymptotically approaching the $k$ times trace of the original CRLB.
This confirms the asymptotic efficiency of the proposed RLS-SA algorithm.

\subsection{Comparison 1: MSE comparison with the quantized recursive identification algorithm for the ARX system}

In \citet{csaji2012recursive}, the researchers have addressed the quantized identification problem for ARX systems with stochastic inputs that are independent and identically distributed, employing a dynamic quantizer in their approach. 
For our comparative analysis, we consider an ARX system as:
\begin{align*}
y_{k} + a_{1}y_{k-1}  =  b_{1}u_{k-1} + d_{k},\quad k \geq 1,
\end{align*}
where the system input $u_{k}$ follows $\mathcal{N}(0,1)$; the unknown parameter $\theta = [a_{1},b_{1}]^{T}=[0.2,1]^{T}$; the observation noise $d_{k}$ follows $\mathcal{N}(0,1)$.

We apply the RLS-SA algorithm with initial conditions \( P_{0} = I_{2}/10 \), \( \hat{\theta}^{\text{RLS}}_{0} = [0,0]^{T} \), and \( \hat{\theta}_{0} = [0,0]^{T} \), where the step-size coefficients are set as \( \alpha = 0.95 \) and \( \beta_{k} \equiv 1 \). Additionally, to facilitate comparison, we apply the sign-error type algorithm proposed by \citet{csaji2012recursive} to address the identification problem. 

The simulation is repeated 500 times. Figure \ref{fig_e1} shows that the $k$ times MSE of the sign-error type algorithm is larger than that of the RLS-SA algorithm, indicating that the performance of the proposed algorithm is superior.

\begin{figure}[h]
\noindent\includegraphics[width=0.48\textwidth]{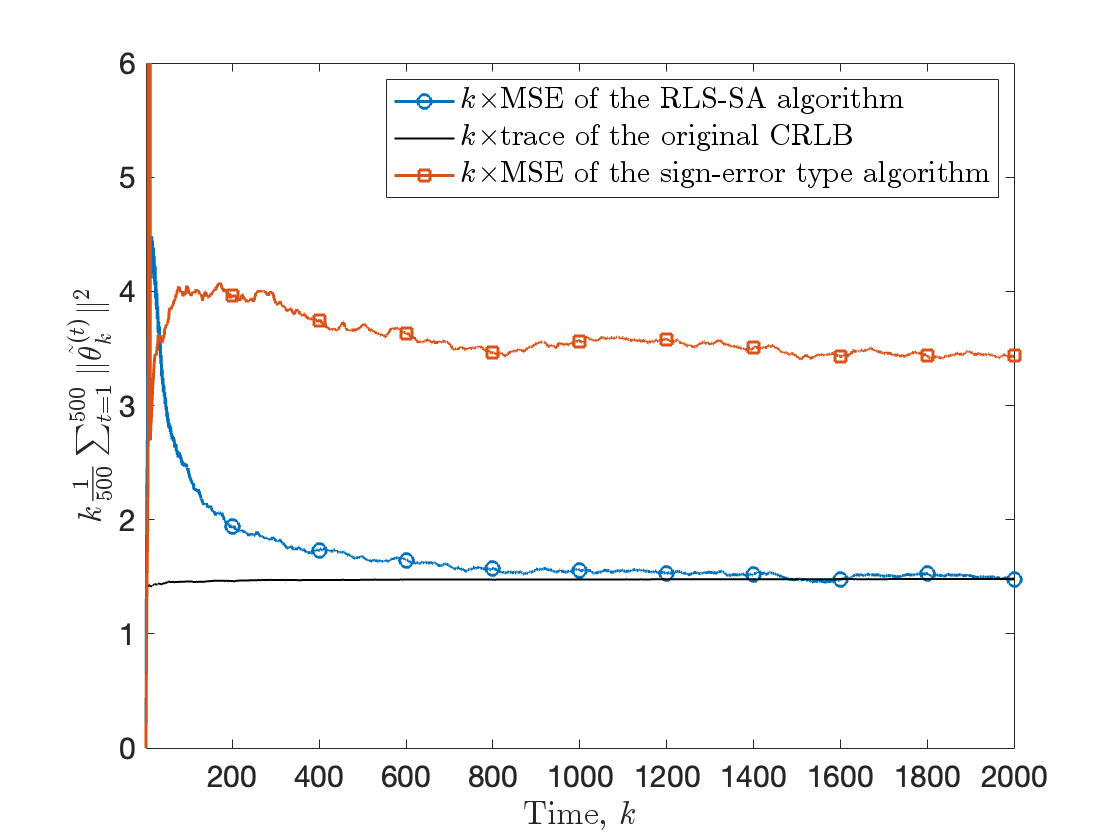}
\caption{Trajectories of $k$ times MSE by the RLS-SA algorithm and the sign-error type algorithm \citep{csaji2012recursive}}
\label{fig_e1}
\end{figure}

\subsection{Comparison 2: MSE comparison with the asymptotically efficient algorithms using a fixed quantizer}

In \citet{yin2007asymptotically} and \citet{wang2024asymptotically}, the researchers have addressed the quantized identification problem for FIR systems with deterministic inputs.
For our comparative analysis, we consider an FIR system as:
\begin{align*}
y_{k} =  b_{1}u_{k-1} + b_{2}u_{k-2} + b_{3}u_{k-3} + d_{k},\quad k \geq 1,
\end{align*}
where the system input $u_{k} = \cos(2\pi k/3)$; the unknown parameter $\theta = [b_{1},b_{2},b_{3}]^{T}=[0.7,-0.3,0.2]^{T}$; the observation noise $d_{k}$ follows $\mathcal{N}(0,4)$.

We apply the RLS-SA algorithm with initial conditions \( P_{0} = I_{3}/10 \), \( \hat{\theta}^{\text{RLS}}_{0} = [0,0,0]^{T} \), and \( \hat{\theta}_{0} = [0,0,0]^{T} \), where the step-size coefficients are set as \( \alpha = 0.95 \) and \( \beta_{k} \equiv 1 \). Additionally, we use the empirical measure method proposed by \citet{yin2007asymptotically} and the IBID algorithm proposed by \citet{wang2024asymptotically} as alternative estimation methods, both considered asymptotically optimal with a fixed quantizer,  with the threshold set to zero.

The simulation is repeated 500 times. Fig. \ref{fig_f1} compares the MSE performance. 
The $k$ times MSE trajectories for both the empirical measure method and the IBID algorithm approach the $k$ times trace of the CRLB for the fixed threshold, as shown in \citet{yin2007asymptotically, wang2024asymptotically}. 
Notably, the $k$ times trace of the CRLB for the fixed threshold is greater than that for the original system output, which indicates that the proposed algorithm outperforms those in \citet{yin2007asymptotically, wang2024asymptotically}.

\begin{figure}[h]
\noindent\includegraphics[width=0.48\textwidth]{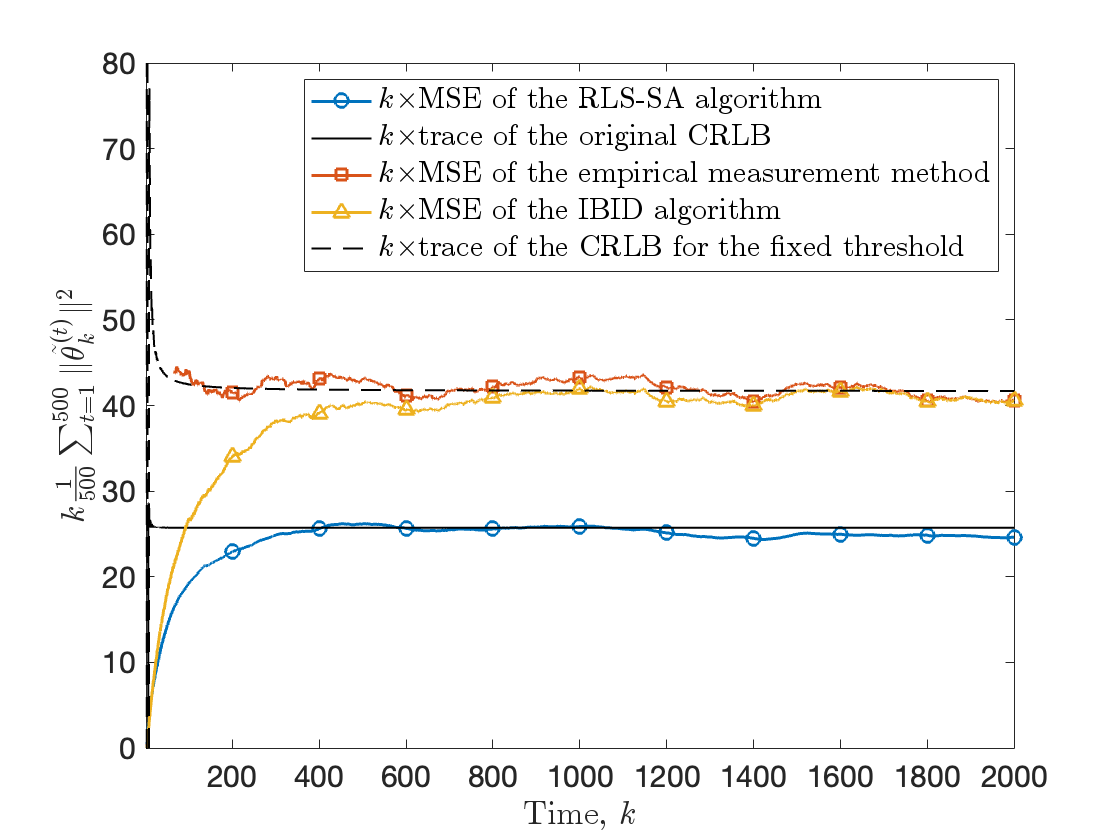}
\caption{Trajectories of $k$ times MSE by the RLS-SA algorithm, the empirical measure method \citep{yin2007asymptotically}, and the IBID algorithm \citep{wang2024asymptotically}}
\label{fig_f1}
\end{figure}

\subsection{Comparison 3: MSE comparison with the asymptotically efficient algorithm using an adaptive quantizer} 

In \citet{you2015recursive}, the researcher has addressed the quantized identification problem for linear systems with stochastic inputs that are independent and identically distributed, employing a dynamic quantizer in their approach. 
For our comparative analysis, we consider a linear system as:
\begin{align*}
y_{k} =  b_{1}u_{k-1} + d_{k},\quad k \geq 1,
\end{align*}
where the system input $u_{k}$ follows a uniform distribution in the range $[-10,10]$; the unknown parameter $\theta = b_{1}=0.5$; the observation noise $d_{k}$ follows $\mathcal{N}(0,1)$.

We apply the RLS-SA algorithm with initial conditions \( P_{0} = 0.1 \), \( \hat{\theta}^{\text{RLS}}_{0} = 0 \), and \( \hat{\theta}_{0} = 0 \), where the step-size coefficients are set as \( \alpha = 0.95 \) and \( \beta_{k} \equiv 1 \). Additionally, we use the optimal Newton-based estimator proposed by \citet{you2015recursive} to provide an alternative estimation, considered asymptotically optimal with adaptive thresholds.

The simulation is repeated 500 times. 
Fig. \ref{fig_e} compares the MSE performance. 
The trajectory of the sample $k$ times MSE for the RLS-SA scheme asymptotically approaches the $k$ times trace of the original CRLB, while the $k$ times MSE trajectory of the optimal Newton-based estimator reaches the $k$ times trace of the CRLB for the adaptive threshold. 

\begin{figure}[h]
\noindent\includegraphics[width=0.48\textwidth]{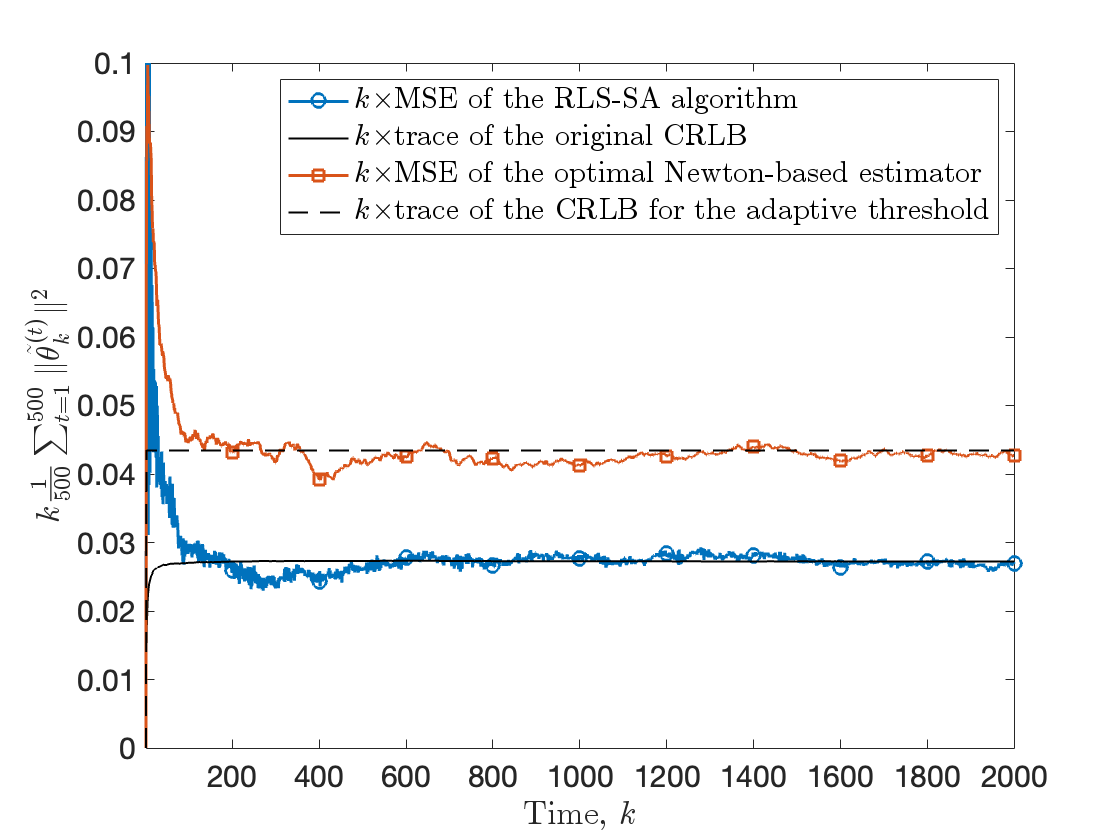}
\caption{Trajectories of $k$ times MSE by the RLS-SA algorithm and the optimal Newton-based estimator \citep{you2015recursive}}
\label{fig_e}
\end{figure}

Notably, Figure \ref{fig_f} demonstrates that the ratio of the MSE of the RLS-SA algorithm to that of the optimal Newton-based estimator converges to $2/\pi \approx 0.64$.
It suggests that, compared to existing identification algorithms under one-bit data, the RLS-SA algorithm reduces the asymptotic MSE by at least $1-2/\pi \approx 36\%$.
Furthermore, since the adaptive threshold proposed by \citet{you2015recursive} converges to the optimal adaptive threshold \citep{wang2024threshold}, this indicates that the original CRLB is $2/\pi \approx 0.64$ times that for the optimal adaptive threshold, which aligns with the analysis presented in \citet{you2015recursive} and \citet{wang2024threshold}. 

\begin{figure}[h]
\noindent\includegraphics[width=0.48\textwidth]{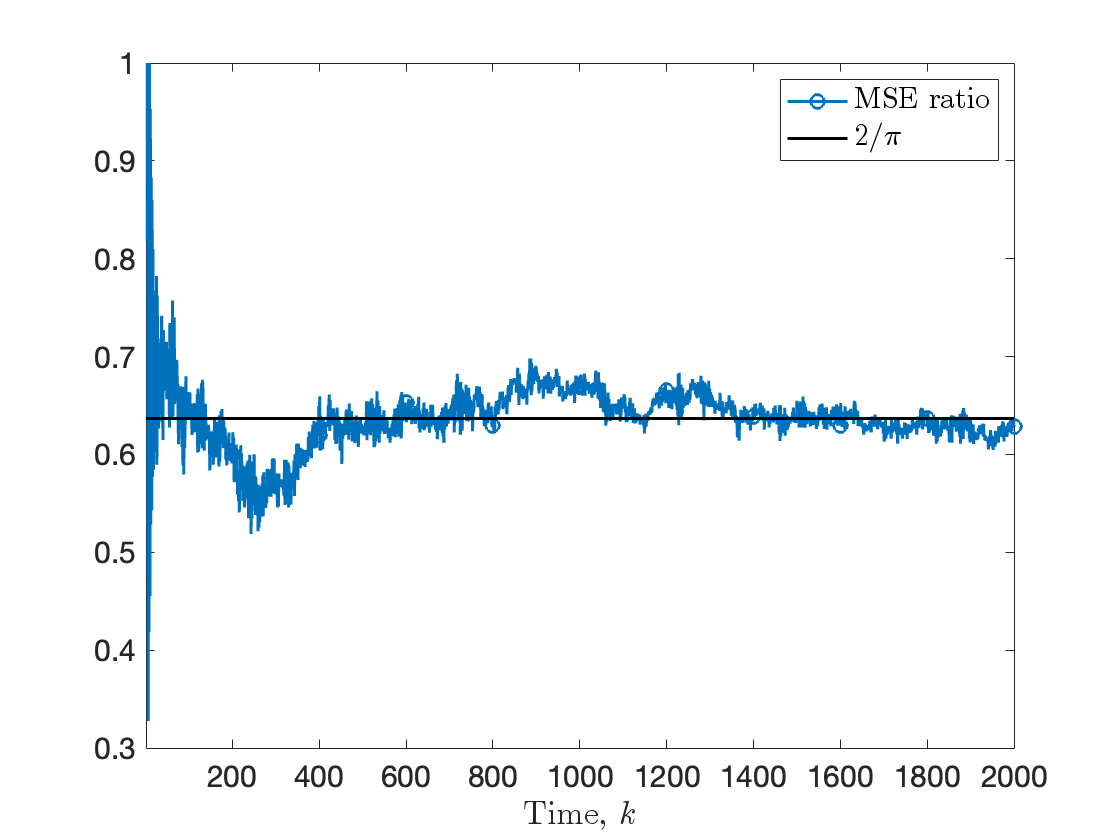}
\caption{Ratio of the MSE by the RLS-SA algorithm and the optimal Newton-based estimator \citep{you2015recursive}}
\label{fig_f}
\end{figure}

\section{Conclusion}\label{sec_f}
This paper proposes an asymptotically efficient recursive identification algorithm for ARX systems under one-bit communications.  
We design a novel quantizer that utilizes both current and historical system outputs and inputs to extract critical parameter information before quantization.
The remote estimator recursively identifies the system using only one-bit data without direct access to the system inputs and outputs.
The remote estimate achieves asymptotic normality, with its error covariance matrix converging to the original CRLB, confirming its asymptotic efficiency. 

Here, we give three topics for future research. 
First, extending ARX systems to more general models, such as autoregressive moving-average systems with exogenous inputs.
Second, will there be any improvement in estimation performance if the parametric knowledge is extracted by another estimation method, such as the instrumental variables estimation method?
Third, an important extension is to investigate the impact of channel imperfections, such as additive noise or packet drops, on the system identification performance.

\appendix

\section{Lemmas} \label{aap_a}

\begin{lemmax}\label{lemma_aaa222} (Theorem 4.5.2 in \citet{Chung2001}) If the random process $\{X_{k}\}$ converges to $X$ in distribution, and for some $p > 0$, $\sup_{k}\mathbb{E}[\Vert X_{k} \Vert^{p}] < \infty$, then, for each $r < p$,
\[
\lim\limits_{k \to \infty} \mathbb{E}[\Vert X_{k} \Vert^{r}] = \mathbb{E}[\Vert X \Vert^{r}] < \infty.
\]
\end{lemmax}  

\begin{lemmax}\label{lemma_b22} (Corollary 2.1 in \citet{Liu2025})
Under Assumptions \ref{ass_a}, for any positive integer $\gamma$,
\begin{align}
\label{lemma_b_4} & \sup_{k}\mathbb{E}\left[\Vert \phi_{k} \Vert^{\gamma}\right] < \infty.
\end{align}
\end{lemmax}

\begin{lemmax}\label{lemma_b21}
Under Assumptions \ref{ass_a}, for any positive integer $\gamma$ and any positive constant $\mu$,
\begin{align}
\label{lemma_b_3} & \mathbb{P}\left(\left\Vert \phi_{k+1} \right\Vert > k^{\mu} \right) = O\left(\frac{1}{k^{\gamma}}\right).
\end{align}
\end{lemmax}

\begin{pf}
By DeMorgan's Law \citep{casella2001statistical} and Boole's Inequality \citep{casella2001statistical}, we have
\begin{align}\label{uuuqqq}
& \mathbb{P}\left( \left\Vert \phi_{k+1} \right \Vert > k^{\mu}\right) \\
= & \mathbb{P}\left( \sum_{i=0}^{m-1}y_{k-i}^{2} + \sum_{j=0}^{n-1}u_{k-j}^{2} > k^{2\mu}\right) \nonumber \\
\leq & \sum_{i=0}^{m-1}   \mathbb{P}\left( y_{k-i}^{2} > \frac{k^{2\mu}}{m+n} \right)\nonumber  + \sum_{j=0}^{n-1}   \mathbb{P}\left( u_{k-j}^{2} > \frac{k^{2\mu}}{m+n} \right).
\end{align}
We consider $\mathbb{P}( y_{k}^{2} > k^{2\mu}/(m+n))$ first.
Define 
$w_{k} \triangleq \sum_{i=1}^{\infty}  f_{i}^{(1)}(k)r_{k-i}$ and  $v_{k}  \triangleq \sum_{i=0}^{\infty} f_{i}^{(2)}(k)e_{k-i}.$
Then, from Assumptions \ref{ass_a}, $\{w_{k}\}_{k=1}^{\infty}$ is bounded and deterministic, $\{v_{k}\}_{k=0}^{\infty}$ is stochastic, and $y_{k} = w_{k} + v_{k}.$
By the $C_{r}$-inequality \citep{Loeve1977}, we obtain $\mathbb{P}( y_{k}^{2} > k^{2\mu}/(m+n)) \leq \mathbb{P}( 2v_{k}^{2} + 2 w_{k}^{2}  > k^{2\mu}/(m+n)) = \mathbb{P}( v_{k}^{2} > k^{2\mu}/(2(m+n)) - w_{k}^{2}).$

Since $\{w_{k}\}_{k=1}^{\infty}$ is a bounded sequence, for sufficiently large $k$, we have $\mathbb{P}( v_{k}^{2} > k^{2\mu}/(2(m+n)) - w_{k}^{2}) \leq \mathbb{P}( v_{k}^{2} > k^{2\mu}/(3(m+n)) )$.

Note that $v_{k}^{2\lceil \gamma/\mu \rceil}$ can be expressed as $v_{k}^{2\lceil \gamma/\mu \rceil} = \sum_{i_{1}=0}^{\infty}\cdots\sum_{i_{2\lceil \gamma/\mu \rceil}=0}^{\infty} \prod_{t=1}^{2\lceil \gamma/\mu \rceil}f_{i_{t}}^{(2)}(k)e_{k-i_{t}}.$
Since $\{e_{k}\}_{k=0}^{\infty}$ is a sequence of independent variables with bounded moments of any positive integer order, for each $i_{1},i_{2},\ldots, i_{2\lceil \gamma/\mu \rceil} \in \{ 0,1,2,\ldots \}$, by H\"{o}lder's Inequality \citep{Loeve1977}, we have $ \vert \mathbb{E}[\prod_{t=1}^{2\lceil \gamma/\mu \rceil}e_{k-i_{t}}]\vert \leq  \prod_{t=1}^{2\lceil \gamma/\mu \rceil} (\mathbb{E}[e_{k-i_{t}}^{2\lceil \gamma/\mu \rceil}])^{1/(2\lceil \gamma/\mu \rceil)} = O(1).$

Then, from Assumptions \ref{ass_a}, it holds that $\mathbb{E}[v_{k}^{2\lceil \gamma/\mu \rceil}] 
= \sum_{i_{1}=0}^{\infty}\cdots\sum_{i_{2\lceil \gamma/\mu \rceil}=0}^{\infty}((  \prod_{t=1}^{2\lceil \gamma/\mu \rceil}f_{i_{t}}^{(2)}(k))\mathbb{E}[\prod_{t_2=1}^{2\lceil \gamma/\mu \rceil}e_{k-i_{t_2}}] ) \\
= O( \sum_{i_{1}=0}^{\infty}\cdots\sum_{i_{2\lceil \gamma/\mu \rceil}=0}^{\infty}\prod_{t=1}^{2\lceil \gamma/\mu \rceil}f_{i_{t}}^{(2)}(k) ) \\
=  O( (\sum_{i=0}^{\infty}\vert f_{i}^{(2)}(k)\vert)^{2\lceil \gamma/\mu \rceil} ) 
=  O( 1 ).$
By the Markov inequality \citep{Loeve1977}, one can get that $\mathbb{P}( v_{k}^{2} > k^{2\mu}/(3(m+n))) \leq (3^{\lceil \gamma/\mu \rceil}(m+n)^{\lceil \gamma/\mu \rceil}\mathbb{E}[v_{k}^{2\lceil \gamma/\mu \rceil}])/k^{2\lceil \gamma/\mu \rceil\mu} \\ = O(1/k^{2\gamma}) = O(1/k^{\gamma})$, which implies that $\mathbb{P}( y_{k}^{2} > k^{2\mu}/(m+n) ) = O(1/k^{\gamma}).$

Similarly, for each $i \in \{ 2,3,4,\ldots,m-1 \}$ and $j \in \{ 0,1,2,\ldots,n-1 \}$, it can be proved that $\mathbb{P}( y_{k-i}^{2} > k^{2\mu}/(m+n)) = O(1/k^{\gamma})$ and $\mathbb{P}( u_{k-j}^{2} > k^{2\mu}/(m+n)) = O(1/k^{\gamma})$. 
Then, (\ref{uuuqqq}) implies (\ref{lemma_b_3}) holds. {\hfill $\qed$}
\end{pf}

\begin{lemmax}\label{lemma_b12} (Corollary 3.1 in \citet{Liu2025})
Under Assumptions \ref{ass_a}, for any positive integer $\gamma$ and positive constant $\epsilon$, 
\begin{align}
\label{lemma_b_1} & \mathbb{P}\left(\left\Vert  \bar{\mathbb{E}}\left[ \{ \phi_{k}\phi_{k}^{T}\}_{k=1}^{\infty}\right] - \frac{P_{k}^{-1}}{k} \right\Vert > \varepsilon\right) = O\left(\frac{1}{k^{\gamma}}\right), \\
\label{lemma_b_2}  & \mathbb{P}\left(\left\Vert P_{k} \right\Vert > \frac{2}{k} \Vert \left(\bar{\mathbb{E}}\left[ \{ \phi_{k}\phi_{k}^{T}\}_{k=1}^{\infty}\right]\right)^{-1} \Vert\right) = O\left(\frac{1}{k^{\gamma}}\right).
\end{align}
\end{lemmax}

\section{Properties of the RLS algorithm} \label{aap_b}

\begin{lemmax}\label{lemma_aaa444} (Theorem 1 in \citet{Guo1986}, Theorems 3.2 and 3.3 in \citet{Liu2025})
Under Assumptions \ref{ass_a}-\ref{ass_c}, 
\begin{align}
& \left\Vert \tilde{\theta}_{k}^{\text{RLS}} \right\Vert = O\left(\sqrt{\frac{\log (1+\sum_{l=1}^{k}\Vert \phi_{l} \Vert^{2})}{\lambda_{\min}(P_{k}^{-1})}}\right) \quad  \mathrm{a.s.}
\end{align}
where $\tilde{\theta}_{k}^{\text{RLS}} \triangleq \hat{\theta}_{k}^{\text{RLS}} - \theta$ is the RLS estimation error.

In addition, for any positive integer $p$, 
\begin{align}
\label{lemma_e_1} & \mathbb{E}\left[ \left\Vert \tilde{\theta}_{k}^{\text{RLS}} \right\Vert^{p} \right] = O\left(\frac{1}{k^{p/2}}\right).
\end{align}
Furthermore, 
\begin{align}
\label{lemma_d_1} & \sqrt{k}\tilde{\theta}_{k}^{\text{RLS}} \xrightarrow{d}  \mathcal{N}\left(0, \bar{\Sigma}_{\mathrm{CR}}  \right),
\end{align}
where 
$
\bar{\Sigma}_{\mathrm{CR}} = \lim\limits_{k \to \infty}k\Sigma_{\mathrm{CR}}(k) = \delta_{d}^{2}(\bar{\mathbb{E}}[ \{ \phi_{k}\phi_{k}^{T}\}_{k=1}^{\infty}])^{-1},
$ and 
$\Sigma_{\mathrm{CR}}(k) = \delta_{d}^{2}(\sum_{l=1}^{k}\mathbb{E}[\phi_{l}\phi_{l}^{T}])^{-1}$ is the original CRLB.
\end{lemmax}

\begin{lemmax}\label{lemma_c}
Under Assumptions \ref{ass_a}-\ref{ass_c}, 
\begin{align}
\label{lemma_c_1} \left\Vert \tilde{\theta}_{k}^{\text{RLS}} \right\Vert = O\left(\sqrt{\frac{\log k}{k}}\right) \quad  \mathrm{a.s.}
\end{align}
\end{lemmax}

\begin{pf}
By Lemma \ref{lemma_b12} in Appendix \ref{aap_a} and Borel-Cantelli lemma \citep{ash2014real}, one can get
\begin{align}\label{Pk13a}
\lim\limits_{k \rightarrow \infty} \left\Vert \bar{\mathbb{E}}\left[ \{\phi_{k}\phi_{k}^{T}\}_{k=1}^{\infty}\right] - \frac{P_{k}^{-1}}{k}\right\Vert
=  0 \quad  \mathrm{a.s.}
\end{align}
which indicates that
\begin{align}\label{Pk13aa1}
\lim\limits_{k \rightarrow \infty} \frac{\lambda_{\min}(P_{k}^{-1})}{k} = \lambda_{\min}\left(\bar{\mathbb{E}}\left[ \{ \phi_{k}\phi_{k}^{T}\}_{k=1}^{\infty}\right]\right) \quad  \mathrm{a.s.}
\end{align}
Besides, by the Woodbury matrix identity \citep{ljung1987theory}, we have $P_{k} = (\sum_{l=1}^{k} \phi_{l}\phi_{l}^{T} + P_{0}^{-1})^{-1}.$
Hence, $\bar{\mathbb{E}}[ \{ \phi_{k}\phi_{k}^{T}\}_{k=1}^{\infty}] - P_{k}^{-1}/k = \bar{\mathbb{E}}[ \{ \phi_{k}\phi_{k}^{T}\}_{k=1}^{\infty}]  - \sum_{l=1}^{k}\phi_{l}\phi_{l}^{T}/k - P_{0}^{-1}/k.$
Then, by (\ref{Pk13a}), it holds that
\[
\lim\limits_{k \rightarrow \infty} \left\Vert   \bar{\mathbb{E}}\left[ \{ \phi_{k}\phi_{k}^{T}\}_{k=1}^{\infty}\right] - \frac{1}{k}\sum_{l=1}^{k}\phi_{l}\phi_{l}^{T}\right\Vert
=  0 \quad  \mathrm{a.s.}
\]
which implies that
\begin{align}\label{a6sa6}
\lim\limits_{k \rightarrow \infty}\frac{1}{k}\sum_{l=1}^{k}\left\Vert\phi_{l}\right\Vert^2 -  \bar{\mathbb{E}}\left[ \{ \phi_{k}^{T}\phi_{k}\}_{k=1}^{\infty}\right] 
=  0 \quad  \mathrm{a.s.}
\end{align}
Thus, by (\ref{Pk13aa1}), (\ref{a6sa6}), and Lemma \ref{lemma_aaa444} in Appendix \ref{aap_b},  one can get (\ref{lemma_c_1}) holds. {\hfill $\qed$}

\end{pf}

\begin{lemmax}\label{lemma_a111}
Under Assumptions \ref{ass_a}-\ref{ass_c}, for any $0 < \gamma < 1$,
\begin{align}\label{ttt222}
\left\Vert \hat{\theta}_{k}^{\text{RLS}} - \hat{\theta}_{k-1}^{\text{RLS}} \right\Vert = O\left(\frac{1}{k^\gamma}\right)  \quad \mathrm{a.s.} 
\end{align}
\end{lemmax}

\begin{pf}
By Lemma \ref{lemma_b12} in Appendix \ref{aap_a} and Borel-Cantelli Lemma \citep{ash2014real}, we obtain 
\[
\left\Vert P_{k} \right\Vert = O\left(\frac{1}{k}\right) \quad \mathrm{a.s.}
\]
Besides, since $\sup_{k}\mathbb{E}[\vert d_{k} \vert^{4/(1-\gamma)}]<\infty$, by the Markov Inequality \citep{Loeve1977}, one can get 
%\[
%\mathbb{P}\left(\vert d_{k}\vert > k^{(1-\gamma)/2}\right) \leq \frac{1}{k^2}\mathbb{E}\left[\vert d_{k} \vert^{4/(1-\gamma)}\right] = O\left(\frac{1}{k^2}\right).
%\]
$\mathbb{P}(\vert d_{k}\vert > k^{(1-\gamma)/2}) \leq \mathbb{E}[\vert d_{k} \vert^{4/(1-\gamma)}]/k^{2} = O(1/k^2).$
Then, by the Borel-Cantelli lemma \citep{ash2014real}, we have
\[
\left\vert d_{k} \right\vert =  O\left(k^{(1-\gamma)/2}\right) \quad \mathrm{a.s.}
\]
Similarly, by Lemma \ref{lemma_b22} in Appendix \ref{aap_a}, it holds that
\[
\left\Vert \phi_{k} \right\Vert = O\left(k^{(1-\gamma)/2}\right) \quad \mathrm{a.s.}
\]
Furthermore, one can get 
%\[
%\hat{\theta}_{k}^{\text{RLS}} = \hat{\theta}_{k-1}^{\text{RLS}} + a_{k}P_{k-1}\phi_{k}d_{k} - a_{k}P_{k-1}\phi_{k}\phi_{k}^{T}\tilde{\theta}_{k-1}^{\text{RLS}}
%\]
$\hat{\theta}_{k}^{\text{RLS}} = \hat{\theta}_{k-1}^{\text{RLS}} + a_{k}P_{k-1}\phi_{k}d_{k} - a_{k}P_{k-1}\phi_{k}\phi_{k}^{T}\tilde{\theta}_{k-1}^{\text{RLS}}$.
Hence, by Lemma \ref{lemma_c} in Appendix \ref{aap_b} and 
%\begin{align*}
%\Vert \hat{\theta}_{k}^{\text{RLS}} - \hat{\theta}_{k-1}^{\text{RLS}} \Vert \leq & \vert a_{k} \vert \Vert P_{k-1} \Vert \Vert \phi_{k} \Vert \vert d_{k} \vert \\ & + \vert a_{k} \vert \Vert P_{k-1} \Vert \Vert \phi_{k} \Vert^{2} \Vert \tilde{\theta}_{k-1}^{\text{RLS}}\Vert,
%\end{align*}
$\Vert \hat{\theta}_{k}^{\text{RLS}} - \hat{\theta}_{k-1}^{\text{RLS}} \Vert \leq \vert a_{k} \vert \Vert P_{k-1} \Vert \Vert \phi_{k} \Vert \vert d_{k} \vert + \vert a_{k} \vert \Vert P_{k-1} \Vert \Vert \phi_{k} \Vert^{2} \Vert \tilde{\theta}_{k-1}^{\text{RLS}}\Vert$,
we have (\ref{ttt222}) holds. {\hfill $\qed$}
\end{pf}

\end{document}